\documentclass[final,notitlepage,12pt,reqno,tbtags]{amsart}
\usepackage{graphicx}

\usepackage{calc}
\usepackage{url}
\usepackage{mathrsfs}
\newlength{\depthofsumsign}
\makeatletter
\let\I\@undefined
\makeatother
\usepackage{geometry}
\usepackage{indentfirst}
\usepackage[normalem]{ulem}
\usepackage{float}
\usepackage{amsthm}

\usepackage{enumitem}[2011/09/28]
\setenumerate{align=left, leftmargin=0pt,labelsep=.5em, labelindent=0\parindent,listparindent=\parindent,itemindent=*}
\usepackage{array}
\usepackage{empheq}
\usepackage{natbib}
\usepackage[all]{xy}

\usepackage{extarrows}

\usepackage{caption}[2012/02/19]

\usepackage{longtable,lscape}

\usepackage{fouriernc}
\usepackage{fourier}
\usepackage{amssymb,bm,amsmath}
\usepackage{amsfonts}

\usepackage[OT2,T1,T2A]{fontenc}

\newcommand{\textcyr}[1]{
{\fontencoding{OT2}\fontfamily{wncyr}\selectfont
\hyphenchar\font=-1\relax#1}}

\usepackage[english,french,german,russian]{babel}
\usepackage{appendix}

\DeclareMathOperator{\IKM}{\mathbf{IKM}}

\DeclareMathOperator{\D}{d}
\DeclareMathOperator{\I}{Im}
\DeclareMathOperator{\R}{Re}

\bibpunct{[}{]}{,}{n}{}{;}

\def\eor{\hfill$ \square$}

\theoremstyle{plain}
\newtheorem{theorem}{Theorem}[subsection]
\newtheorem{proposition}[theorem]{Proposition}
\newtheorem{lemma}[theorem]{Lemma}

\newenvironment{remark}[1][Remark]{\begin{trivlist}
\item[\hskip \labelsep {\bfseries #1}]}{\end{trivlist}}

\newbox\shell
\newcommand{\dia}[2]{\setbox\shell=\hbox{\begin{picture}(180,120)(-90,-60)#1
\put(-90,-60){\makebox(180,120)[b]{\large #2}}\end{picture}}\dimen0=\ht
\shell\multiply\dimen0by7\divide\dimen0by16\raise-\dimen0\box\shell\hfill}

\newcommand{\vtx}{\circle*{10}}

\theoremstyle{definition}

\numberwithin{equation}{subsection}

\usepackage{color}

\begin{document}

\pagenumbering{roman}
\selectlanguage{english}
\title[Sunrise at 4 loops]{On Laporta's 4-loop sunrise formulae}
\author{Yajun Zhou}
\address{Program in Applied and Computational Mathematics (PACM), Princeton University, Princeton, NJ 08544} \email{yajunz@math.princeton.edu}\curraddr{\textsc{Academy of Advanced Interdisciplinary Studies (AAIS), Peking University, Beijing 100871, P. R. China}}\email{yajun.zhou.1982@pku.edu.cn}
\date{\today}

\maketitle

\begin{abstract}
   We prove Laporta's conjecture\begin{align*}&\int_0^\infty\frac{\mathrm d\, x_1}{x_1}\int_0^\infty\frac{\mathrm d\, x_2}{x_2}\int_0^\infty\frac{\mathrm d\, x_3}{x_3}\int_0^\infty\frac{\mathrm d\, x_4}{x_4}\frac{1}{\left(1+\sum^4_{k=1}x_k\right)\left(1+\sum^4_{k=1}\frac{1}{x_{k}} \right)-1}\\={}&\frac43
\int_{0}^\pi\mathrm d\, \phi_1
\int_{0}^\pi\mathrm d\, \phi_2\int_{0}^\pi\mathrm d\, \phi_3
\int_{0}^\pi\mathrm d\, \phi_4\frac{1}{4-\sum_{k=1}^4\cos \phi_k},
\end{align*} which relates the 4-loop sunrise diagram in 2-dimensional quantum field theory to Watson's integral for 4-dimensional hypercubic lattice. We also establish several related integral identities proposed by Laporta, including a reduction of the 4-loop sunrise diagram to special values of Euler's gamma function and generalized hypergeometric series:\begin{align*}
\frac{4 \pi ^{5/2}}{\sqrt{3}}\left\{ \frac{\sqrt{3} }{2^6 }\left[\frac{\Gamma \left(\frac{1}{3}\right)}{\sqrt{\pi}}\right]^9\, _4F_3\left(\left. \begin{array}{c}\frac{1}{6},\frac{1}{3},\frac{1}{3},\frac{1}{2}\\[4pt]\frac{2}{3},\frac{5}{6},\frac{5}{6}\end{array} \right|1\right)-\frac{2^{4}}{3}\left[\frac{\sqrt{\pi}}{\Gamma \left(\frac{1}{3}\right)}\right]^9\, _4F_3\left(\left. \begin{array}{c}\frac{1}{2},\frac{2}{3},\frac{2}{3},\frac{5}{6}\\[4pt]\frac{7}{6},\frac{7}{6},\frac{4}{3}\end{array} \right|1\right) \right\}.
\end{align*}\end{abstract}
\pagenumbering{roman}
\tableofcontents

\clearpage

\pagenumbering{arabic}

\section{Introduction}
\subsection{Laporta's empirical formulae for 4-loop sunrise}

In 2-dimensional quantum field theory,
 the 4-loop sunrise diagram  refers to the following object:\begin{align}\begin{split}
\;\;\;\;\;
\dia{\put(-100,0){\line(1,0){200}}
\put(0,15){\circle{100}}
\put(0,-15){\circle{100}}
\put(50,0){\vtx}
\put(-50,0){\vtx}
}{}\;\;\;={}&2^{4}\int_0^\infty I_0(t)[K_0(t)]^5t\D t\\={}&\end{split}\int_0^\infty\frac{\D x_1}{x_1}\int_0^\infty\frac{\D x_2}{x_2}\int_0^\infty\frac{\D x_3}{x_3}\int_0^\infty\frac{\D x_4}{x_4}\frac{1}{\left(1+\sum^4_{k=1}x_k\right)\left(1+\sum^4_{k=1}\frac{1}{x_{k}} \right)-1}.
\end{align} Here, in the configuration space, the Feynman diagram is represented by a single integral over the variable $t$, which involves modified
 Bessel functions  \begin{align} I_0(t)=\frac{1}{\pi}\int_0^\pi e^{t\cos\theta}\D\theta \quad\text{and}\quad K_0(t)=\int_0^\infty e^{-t\cosh u}\D u;\label{eq:defn_I0K0}\end{align} in the Schwinger parameter space, the Feynman diagram is represented by a quadruple integral over a rational function in the variables $x_1$, $x_2$, $x_3$ and $x_4$. It is a well-established fact (see, for example, \cite[][\S\S9.1--9.2]{Broadhurst2013MZV} or \cite[][\S8]{Vanhove2014Survey}) that  the aforementioned single and quadruple integral representations are equivalent to each other.

Numerical experiments have led to still  more  integral representations for the 4-loop sunrise diagram. Some of these empirical formulae have remained long-standing conjectures.  For example, in 2008 and 2017, Laporta suggested that  \cite[][(72) and (81)]{Laporta2008} \begin{align}
\;\;\;\;\;
\dia{\put(-100,0){\line(1,0){200}}
\put(0,15){\circle{100}}
\put(0,-15){\circle{100}}
\put(50,0){\vtx}
\put(-50,0){\vtx}
}{}\;\;\;=\frac43
\int_{0}^\pi\D \phi_1
\int_{0}^\pi\D \phi_2\int_{0}^\pi\D \phi_3
\int_{0}^\pi\D \phi_4\frac{1}{4-\sum_{k=1}^4\cos \phi_k},\label{eq:LaportaWatson}
\end{align}and    (see \cite[][(69) and (72)]{Laporta2008} and \cite[][(28)]{Laporta:2017okg})\begin{align}\begin{split}
&\;\;\;\;\;
\dia{\put(-100,0){\line(1,0){200}}
\put(0,15){\circle{100}}
\put(0,-15){\circle{100}}
\put(50,0){\vtx}
\put(-50,0){\vtx}
}{}\;\;\;=\frac{4  \sqrt{3} \pi ^3}{27} \int_{0}^1\left[ _2F_1\left(\left.\begin{array}{c}
\frac{1}{3},\frac{2}{3} \\[4pt]
1 \\
\end{array}  \right|x\right)\right]^2\frac{\D x}{\sqrt{1-x}}\\={}&\frac{4 \pi ^{5/2}}{\sqrt{3}}\left\{ \frac{\left[\Gamma \left(\frac{7}{6}\right)\right]^2 \Gamma \left(\frac{1}{3}\right)}{\left[\Gamma \left(\frac{2}{3}\right)\right]^2 \Gamma \left(\frac{5}{6}\right)}\, _4F_3\left(\left. \begin{array}{c}\frac{1}{6},\frac{1}{3},\frac{1}{3},\frac{1}{2}\\[4pt]\frac{2}{3},\frac{5}{6},\frac{5}{6}\end{array} \right|1\right)+\frac{\left[\Gamma \left(\frac{5}{6}\right)\right]^2 \Gamma \left(-\frac{1}{3}\right)}{\left[\Gamma \left(\frac{1}{3}\right)\right]^2 \Gamma \left(\frac{1}{6}\right)}\, _4F_3\left(\left. \begin{array}{c}\frac{1}{2},\frac{2}{3},\frac{2}{3},\frac{5}{6}\\[4pt]\frac{7}{6},\frac{7}{6},\frac{4}{3}\end{array} \right|1\right) \right\}.\end{split}\label{eq:Laporta_pFq}
\end{align} Here, the quadruple integral is a 4-dimensional analog  \cite{GlasserGuttmann1994,Joyce2003,Zucker2011} of a famous problem solved by G. N. Watson \cite{Watson1939},  the (generalized) hypergeometric series is defined by
\begin{align}{_pF_q}\left(\left.\begin{array}{c}
a_{1},\dots,a_p \\[4pt]
b_{1},\dots,b_q \\
\end{array}\right| x\right):=1+\sum_{n=1}^\infty\frac{\prod_{j=1}^p(a_{j})_n}{\prod_{k=1}^q(b_{k})_n }\frac{x^n}{n!}\label{eq:defn_pFq},\end{align}
 with $ (a)_{n}=\prod_{m=0}^{n-1}(a+m)$ being the rising factorial, and  the gamma function is given by \cite[][p.~163, (3)]{SteinII} \begin{align}
\Gamma(s)=\sum^\infty_{n=0}\frac{(-1)^n}{n!(n+s)}+\int_1^\infty e^{-t}t^{s-1}\D t,
\end{align} for $s\in\mathbb C\smallsetminus\mathbb Z_{\leq0} $.

Using the Legendre--Gau{\ss} multiplication formula for Euler's gamma function, one can simplify the ratios of gamma functions in   \eqref{eq:Laporta_pFq}, so that only $ \Gamma\left(\frac13\right)$ is retained in the final presentation \cite[][\S54]{NielsenGamma}. The results are \begin{align}
\frac{\left[\Gamma \left(\frac{7}{6}\right)\right]^2 \Gamma \left(\frac{1}{3}\right)}{\left[\Gamma \left(\frac{2}{3}\right)\right]^2 \Gamma \left(\frac{5}{6}\right)}=\frac{\sqrt{3} }{2^6 }\left[\frac{\Gamma \left(\frac{1}{3}\right)}{\sqrt{\pi}}\right]^9,\quad \frac{\left[\Gamma \left(\frac{5}{6}\right)\right]^2 \Gamma \left(-\frac{1}{3}\right)}{\left[\Gamma \left(\frac{1}{3}\right)\right]^2 \Gamma \left(\frac{1}{6}\right)}=-\frac{2^{4}}{3}\left[\frac{\sqrt{\pi}}{\Gamma \left(\frac{1}{3}\right)}\right]^9,
\end{align}hence the formula stated in the abstract. Hereafter, we will always implement such a  policy of reducing  gamma factors.
\subsection{Strategies for proving Laporta's formulae and their analogs}
In \S\ref{sec:IKM15x} of this work, we verify both  \eqref{eq:LaportaWatson} and \eqref{eq:Laporta_pFq}, through manipulations of certain infinite series, along with applications of  previous results on Watson integrals \cite{GlasserMontaldi1993,Guttmann2009} and Feynman diagrams \cite{BBBG2008,HB1}.
A major tool in our proof is the Meijer $G$-function, which is defined by an integral of Mellin--Barnes type:\begin{align}
G^{m,n}_{p,q}\left(z\left|
\begin{array}{c}
 a_{1},\dots ,a_p \\[4pt]
 b_{1},\dots,b_q \\
\end{array}
\right.\right):=\frac{1}{2\pi i}\resizebox{1.3\width}{1.5\height}{$\displaystyle\int$}_{\hspace{-0.6em}C}\frac{\prod_{j=1}^n\Gamma(1-a_{j}-s)\prod_{k=1}^m\Gamma(b_{k}+s)}{\prod_{j=n+1}^p\Gamma(a_{j}+s)\prod_{k=m+1}^q\Gamma(1-b_{k}-s)}\frac{\D s}{z^s}.\label{eq:MeijerG_defn}
\end{align}Here,  the contour $C$ is chosen such that the right-hand side of the equation above represents the sum over the residues of \begin{align}
-\frac{\prod_{j=1}^n\Gamma(1-a_{j}-s)\prod_{k=1}^m\Gamma(b_{k}+s)}{\prod_{j=n+1}^p\Gamma(a_{j}+s)\prod_{k=m+1}^q\Gamma(1-b_{k}-s)}\frac{1}{z^{s}}
\end{align}at all the poles in  $ \prod_{j=1}^n\Gamma(1-a_{j}-s)$.  Empty products, by convention, are equal to  unity.

In \S\ref{sec:sunrise_analogs}, we study some analogs of Laporta's 4-loop sunrise formulae, namely, hypergeometric representations for several Bessel moments\begin{align}
\IKM(a,b;n):=\int_0^\infty [I_0(t)]^a[K_0(t)]^bt^n\D t
\end{align}satisfying $ a+b=6$ and $ a,b\in\mathbb Z_{>0}$.
Extending the techniques in  \S\ref{sec:IKM15x} with Bailey's hypergeometric identity \cite[][(3.4)]{Bailey1932} and Vanhove's differential equations \cite[][\S9]{Vanhove2014Survey},
we prove the following integral evaluations proposed by Laporta \cite[][(29)]{Laporta:2017okg} and Broadhurst (private communication on Nov.~10, 2017):\begin{align}\begin{split}&
\int_0^\infty I_0(t)[K_0(t)]^5t(1-8t^{2})\D t=\frac{7 \pi ^3}{108 \sqrt{3}}\int_{0}^1\left[ _2F_1\left(\left.\begin{array}{c}
-\frac{1}{3},\frac{1}{3} \\[4pt]
1 \\
\end{array}  \right|x\right)\right]^2\frac{\D x}{\sqrt{1-x}}\\={}&\frac{7 \pi ^{5/2}}{36 \sqrt{3}}\left\{\frac{\sqrt{3} }{2^7}\left[\frac{\Gamma \left(\frac{1}{3}\right)}{\sqrt{\pi }}\right]^9\, _4F_3\left(\left. \begin{array}{c}-\frac{1}{2},\frac{1}{6},\frac{1}{3},\frac{4}{3}\\[4pt]-\frac{1}{6},\frac{5}{6},\frac{5}{3}\end{array} \right|1\right)+\frac{5}{7}\frac{2^4}{3}  \left[\frac{\sqrt{\pi }}{\Gamma \left(\frac{1}{3}\right)}\right]^9\, _4F_3\left(\left. \begin{array}{c}-\frac{7}{6},-\frac{1}{2},-\frac{1}{3},\frac{2}{3}\\[4pt]-\frac{5}{6},\frac{1}{6},\frac{1}{3}\end{array} \right|1\right) \right\}\\={}& \frac{81 \sqrt{3} \pi ^3}{2200}\, _7F_6\left(\left. \begin{array}{c}
-\frac{1}{3},\frac{1}{3},\frac{2}{3},\frac{4}{3},\frac{3}{2},\frac{3}{2},\frac{7}{4} \\[4pt]
\frac{3}{4},1,\frac{7}{6},\frac{11}{6},\frac{13}{6},\frac{17}{6} \\
\end{array} \right|1\right),\end{split}\label{eq:IKM151_IKM153_diff}
\end{align}as well as the following identities discovered by Laporta \cite[][(27)]{Laporta:2017okg} and Broadhurst (see \cite[][\S2.2]{Broadhurst2017Paris}, \cite[][\S2.2]{Broadhurst2017CIRM}, \cite[][\S2.1]{Broadhurst2017Higgs}, \cite[][\S3.1]{Broadhurst2017DESY}, \cite[][\S3.1]{Broadhurst2017ESIa}):\begin{align}\begin{split}
\IKM(2,4;1)={}&\frac{\pi^2}{30}\int_0^1{_2F_1}\left(\left.\begin{array}{c}
\frac{1}{3},\frac{2}{3}\ \\[4pt]
1 \\
\end{array}  \right|x\right){_2F_1}\left(\left.\begin{array}{c}
\frac{1}{3},\frac{2}{3}\ \\[4pt]
1 \\
\end{array}  \right|1-x\right)\frac{\D x}{\sqrt{1-x}}\\={}&\frac{3 \pi ^{3/2}}{20}\left\{ \frac{\sqrt{3} }{2^6 }\left[\frac{\Gamma \left(\frac{1}{3}\right)}{\sqrt{\pi}}\right]^9\, _4F_3\left(\left. \begin{array}{c}\frac{1}{6},\frac{1}{3},\frac{1}{3},\frac{1}{2}\\[4pt]\frac{2}{3},\frac{5}{6},\frac{5}{6}\end{array} \right|1\right)+\frac{2^{4}}{3}\left[\frac{\sqrt{\pi}}{\Gamma \left(\frac{1}{3}\right)}\right]^9\, _4F_3\left(\left. \begin{array}{c}\frac{1}{2},\frac{2}{3},\frac{2}{3},\frac{5}{6}\\[4pt]\frac{7}{6},\frac{7}{6},\frac{4}{3}\end{array} \right|1\right) \right\}\\={}&\frac{\pi^{2}}{10}{_4F_3}\left(\left.\begin{array}{c}
\frac{1}{3},\frac{1}{2},\frac{1}{2},\frac{2}{3} \\[4pt]\frac{5}{6},1,\frac{7}{6} \\
\end{array}\right|1\right).\end{split}
\end{align}We also establish a similar result for $\IKM(2,4;1)-8\IKM(2,4;3) $:\begin{align}\begin{split}&
\int_0^\infty [I_0(t)]^{2}[K_0(t)]^4t(1-8t^2)\D t\\={}&\frac{7 \pi ^{3/2}}{60}\left\{\frac{\sqrt{3} }{2^7}\left[\frac{\Gamma \left(\frac{1}{3}\right)}{\sqrt{\pi }}\right]^9\, _4F_3\left(\left. \begin{array}{c}-\frac{1}{2},\frac{1}{6},\frac{1}{3},\frac{4}{3}\\[4pt]-\frac{1}{6},\frac{5}{6},\frac{5}{3}\end{array} \right|1\right)-\frac{5}{7}\frac{2^4}{3}  \left[\frac{\sqrt{\pi }}{\Gamma \left(\frac{1}{3}\right)}\right]^9\, _4F_3\left(\left. \begin{array}{c}-\frac{7}{6},-\frac{1}{2},-\frac{1}{3},\frac{2}{3}\\[4pt]-\frac{5}{6},\frac{1}{6},\frac{1}{3}\end{array} \right|1\right) \right\}\\={}&\frac{9\pi ^2}{550}  \, _4F_3\left(\left.\begin{array}{c}
\frac{2}{3},\frac{4}{3},\frac{3}{2},\frac{5}{2} \\[4pt]2,\frac{13}{6},\frac{17}{6} \\
\end{array}\right|1\right).\end{split}\label{eq:IKM243_hypergeo_repn}
\end{align}During the course of our proof, we also obtain other hypergeometric representations of Bessel moments. For example, we may equate \eqref{eq:IKM151_IKM153_diff} with \begin{align}
\frac{\pi^{3}}{2^{1/3}}\left[\frac{\sqrt{\pi }}{\Gamma \left(\frac{1}{3}\right)}\right]^6&{_5F_4}\left(\left. \begin{array}{c}
-\frac{1}{3},\frac{1}{2},\frac{3}{2},\frac{3}{2},\frac{5}{3} \\[4pt]
\frac{7}{6},\frac{7}{6},\frac{13}{6},3 \\
\end{array} \right|1\right),
\end{align}in view of \eqref{eq:5F4_redn_a}.

Here, we point out that the Bessel moment $ \IKM(1,5;3)$ contributes a term to Laporta's 4-loop perturbative expansion of electron's  $g-2$ in $ (4-\varepsilon)$-dimensional quantum electrodynamics \cite{Laporta2008,Laporta:2017okg}.
The Bessel moments $ \IKM(2,4;1)$ and $ \IKM(2,4;3)$ did not appear in  Laporta's final result, but were indispensable to the following non-linear sum rule for Feynman diagrams:\begin{align}
\det\begin{pmatrix}\IKM(1,5;1) & \IKM(1,5;3) \\
\IKM(2,4;1) & \IKM(2,4;3) \\
\end{pmatrix}=\frac{\pi^{4}}{576}.\label{eq:detN2}
\end{align}The determinant above had been  discovered by Broadhurst--Mellit (see \cite[][(5.7)]{BroadhurstMellit2016} and \cite[][(113)]{Broadhurst2016}) through numerical experiments, before a proof was found \cite[][\S3]{Zhou2017BMdet}.
Plugging the hypergeometric representations of Bessel moments into the Broadhurst--Mellit determinant formula \eqref{eq:detN2}, we obtain  \begin{align}\begin{split}1={}&
\frac{7}{40} \, _4F_3\left(\left. \begin{array}{c}\frac{1}{2},\frac{2}{3},\frac{2}{3},\frac{5}{6}\\[4pt]\frac{7}{6},\frac{7}{6},\frac{4}{3}\end{array} \right|1\right) {_4F_3}\left(\left. \begin{array}{c}-\frac{1}{2},\frac{1}{6},\frac{1}{3},\frac{4}{3}\\[4pt]-\frac{1}{6},\frac{5}{6},\frac{5}{3}\end{array} \right|1\right)\\{}&+\frac{1}{4} \, _4F_3\left(\left. \begin{array}{c}\frac{1}{6},\frac{1}{3},\frac{1}{3},\frac{1}{2}\\[4pt]\frac{2}{3},\frac{5}{6},\frac{5}{6}\end{array} \right|1\right){_4F_3}\left(\left. \begin{array}{c}-\frac{7}{6},-\frac{1}{2},-\frac{1}{3},\frac{2}{3}\\[4pt]-\frac{5}{6},\frac{1}{6},\frac{1}{3}\end{array} \right|1\right). \end{split}
\end{align}
\section{Laporta's formulae for 4-loop sunrise $ \IKM(1,5;1)$\label{sec:IKM15x}}
\subsection{Watson's hypercubic integral and 4-loop sunrise }

We now consider the  4-dimensional Watson integral for the simple cubic lattice:\begin{align}W_{4}^{S}(x):=\frac{1}{4\pi^4}
\int_{0}^\pi\D \phi_1
\int_{0}^\pi\D \phi_2\int_{0}^\pi\D \phi_3
\int_{0}^\pi\D \phi_4\frac{1}{1-\frac{x}{4}\sum_{k=1}^4\cos \phi_k}.
\end{align}
Laporta's conjecture in \eqref{eq:LaportaWatson} essentially says\begin{align}
W_{4}^{S}(1)=\frac{12}{\pi^4}\int_0^\infty I_0(t)[K_0(t)]^5t\D t\equiv\frac{12}{\pi^4}\IKM(1,5;1).\label{eq:LaportaW4(1)}
\end{align}  \begin{proposition}[Watson integrals and Bessel moments]For all $ x\in[0,1]$, we have \begin{align}
W_{4}^{S}(x)=\frac{4}{\pi^2}\int_0^\infty [I_{0}(xt)]^{2} I_0(t)[K_0(t)]^3t\D t,\label{eq:W4x_BM}
\end{align}and this incorporates \eqref{eq:LaportaW4(1)}  as a special case.\end{proposition}\begin{proof}

Following Guttmann \cite[][\S3.1]{Guttmann2009}, we transcribe an identity of Glasser--Montaldi \cite[][(8)]{GlasserMontaldi1993} as follows:\begin{align}\begin{split}&
\frac{1}{4\pi^4}
\int_{0}^\pi\D \phi_1
\int_{0}^\pi\D \phi_2\int_{0}^\pi\D \phi_3
\int_{0}^\pi\D \phi_4\left(\frac{\sum_{k=1}^4\cos \phi_k}{4}\right)^{2n}\\={}&\frac{1}{2^{2(3n+1)}}\frac{(2n)!}{(n!)^2}\sum_{\substack{j,k,\ell,m\in\mathbb Z_{\geq0}\\j+k+\ell+m=n}}\left( \frac{n!}{j!k!\ell !m!} \right)^2,\quad \end{split}
\end{align}where $n$ is a non-negative integer. Meanwhile, from the work of Bailey--Borwein--Broadhurst--Glasser \cite[][\S4.1]{BBBG2008}, we know that \begin{align}
\int_0^\infty I_0(t)[K_0(t)]^3t^{2n+1}\D t=\frac{(n!)^{2}\pi^{2}}{2^{4(n+1)}}\sum_{\substack{j,k,\ell,m\in\mathbb Z_{\geq0}\\j+k+\ell+m=n}}\left( \frac{n!}{j!k!\ell !m!} \right)^2
\end{align} holds for all non-negative integers $n$.
Thus, we may prove \eqref{eq:W4x_BM} by termwise summation, bearing in mind that \begin{align}
[I_{0}(xt)]^{2}=\sum_{n=0}^\infty\frac{(2n)!}{(n!)^4}\left( \frac{xt}{2} \right)^{2n}.
\end{align}

Finally, the integral identity $ \pi^2\int_0^\infty [I_{0}(t)]^{3} [K_0(t)]^3t\D t=3\int_0^\infty I_0(t)[K_0(t)]^5t\D t$ has been proved in \cite[][Lemma 3.1]{HB1}, so  \eqref{eq:LaportaW4(1)}   is recovered. \end{proof}

We note that there have been previous efforts to represent $ W^S_4(1)$ (that is, the 4-loop sunrise diagram, up to a normalizing constant) as single integrals over familiar functions. For example, using Abel transforms, Glasser--Montaldi \cite[][(8), (A13)]{GlasserMontaldi1993} and Glasser--Guttman \cite[][(3)--(4)]{GlasserGuttmann1994} have shown that \begin{align}
W^S_4(1)=\frac{2}{\pi^{3}}\int_0^1\frac{\mathbf K(k_+)\mathbf K(k_-)}{\sqrt{1-x^2}}\D x
\end{align}where \begin{align}
k_\pm^2=\frac{1}{2} \left[1\pm x^2\sqrt{1-\frac{x^2}{4}} -\left(1-\frac{x^2}{2}\right) \sqrt{1-x^2}\right]
\end{align}and \begin{align} \mathbf K(\sqrt{\lambda})=\int_0^{\pi/2}\frac{\D\phi}{\sqrt{1-\lambda\sin^2\smash[b]{\phi}}}=\frac\pi2{_2F_1}\left(\left.\begin{array}{c}
\frac{1}{2},\frac{1}{2} \\[4pt]
1 \\
\end{array}  \right|\lambda\right)\end{align} is the complete elliptic integral of the first kind.
One can also build more recondite single integral representations for $ W_4^S(1)$, whose integrands involve closed-form expressions of the 3-dimensional Watson integral for the simple cubic lattice:\begin{align}
W_{3}^{S}(x):=\frac{1}{3\pi^3}
\int_{0}^\pi\D \phi_1
\int_{0}^\pi\D \phi_2\int_{0}^\pi\D \phi_3
\frac{1}{1-\frac{x}{3}\sum_{k=1}^3\cos \phi_k},
\end{align}  such as the following formulae established  by Joyce--Zucker \cite[][(3.32), (3.42)]{JoyceZucker2005}: \begin{align}
W_{3}^{S}(x)={}&\frac{2-\sqrt{1-x^2}}{3+x^2}\left[{_2F_1}\left(\left.\begin{array}{c}
\frac{1}{8},\frac{3}{8} \\[4pt]
1 \\
\end{array}  \right|\frac{16 x^2[9-5 x^2-(9-x^2) \sqrt{1-x^2}]^{2}}{9 (3+x^2)^4}\right)\right]^2,\\W_3^S(x)={}&\frac{1-9 p^4}{3(1-p)^3 (3 p+1)}\left[{_2F_1}\left(\left.\begin{array}{c}
\frac{1}{2},\frac{1}{2} \\[4pt]
1 \\
\end{array}  \right|\frac{16 p^3}{(1-p)^3 (3 p+1)}\right)\right]^2,
\end{align}where\begin{align}
p=\sqrt{\frac{1-\sqrt{1-\vphantom{\frac11}\smash[tb]{\frac{x^2}{9}}}}{1+\sqrt{1-x^2}}}.
\end{align}

There is another type of integral representation involving hypergeometric integrands, which in turn, is inspired by arithmetic considerations. Let $ \eta(z):=e^{\pi iz/12}\prod_{n=1}^\infty(1-e^{2\pi inz})$ be the Dedekind eta function, defined for complex numbers $z$ with a positive imaginary part.  It was conjectured in \cite[][(111)]{Broadhurst2016} and  proved in \cite[][Theorem 4.2.5]{Zhou2017WEF} that \begin{align}\;\;\;\;\;
\dia{\put(-100,0){\line(1,0){200}}
\put(0,15){\circle{100}}
\put(0,-15){\circle{100}}
\put(50,0){\vtx}
\put(-50,0){\vtx}
}{}\;\;\;=8\pi ^{2}L(f_{4,6},2):=-32\pi^4\int_{0}^{i\infty}f_{4,6}(z)z\D z,\label{eq:spec_L_val}
\end{align}for a weight-4 level-6 modular form $ f_{4,6}(z)=[\eta(z)\eta(2z)\eta(3z)\eta(6z)]^{2}$. Parametrizing  modular forms with hypergeometric functions, as in the proof of  \cite[][Theorems 4.2.5 and 4.2.6]{Zhou2017WEF}, we obtain\begin{align}\begin{split}
W_4^S(1)={}&\frac{\sqrt{3}}{\pi}\int_0^\infty {_2F_1}\left(\left.\begin{array}{c}
\frac{1}{3},\frac{2}{3} \\[4pt]
1 \\
\end{array}  \right|\frac{u^2 (9+u)}{(3+u)^3}\right){_2F_1}\left(\left.\begin{array}{c}
\frac{1}{3},\frac{2}{3} \\[4pt]
1 \\
\end{array}  \right|1-\frac{u^2 (9+u)}{(3+u)^3}\right)\frac{\D u}{(3+u)^2}\\={}&\frac{2\sqrt{3}}{\pi}\int_{-1}^0 {_2F_1}\left(\left.\begin{array}{c}
\frac{1}{3},\frac{2}{3} \\[4pt]
1 \\
\end{array}  \right|\frac{u^2 (9+u)}{(3+u)^3}\right){_2F_1}\left(\left.\begin{array}{c}
\frac{1}{3},\frac{2}{3} \\[4pt]
1 \\
\end{array}  \right|1-\frac{u^2 (9+u)}{(3+u)^3}\right)\frac{\D u}{(3+u)^2}.\end{split}
\end{align}Here, we point out that last integral representation is actually equivalent to a formula of Bailey--Borwein--Broadhurst--Glasser \cite[][(223)]{BBBG2008}
\begin{align}\begin{split}&
\int_0^\infty [I_{0}(t)]^{3} [K_0(t)]^3t\D t\equiv \IKM(3,3;1)\\={}&\frac{8}{\pi}\int_0^{1/3}\frac{y}{(3 y+1) (1-y)^3}\mathbf K\left( \sqrt{\frac{(1-3 y) (1+y)^3}{(1+3 y) (1-y)^3}} \right)\mathbf K\left( \sqrt{\frac{16 y^3}{(1+3 y) (1-y)^3}} \right)\D y,\end{split}
\end{align}according to  $ W_{4}^{S}(1)=\frac{4}{\pi^2}\IKM(3,3;1)$ [cf.~\eqref{eq:W4x_BM}] and Ramanujan's cubic transformations  \cite[][pp.~112--114]{RN5} for elliptic integrals: \begin{align}\begin{split}
 {_2F_1}\left(\left.\begin{array}{c}
\frac{1}{3},\frac{2}{3} \\[4pt]
1 \\
\end{array}  \right|\frac{27 p^2 (1+p)^2}{4 (1+p+p^2)^3}\right)={}&\frac{2}{\pi}\frac{1+p+p^2}{\sqrt{1+2p}}\mathbf K\left( \sqrt{\frac{p^3(2+p)}{1+2p}} \right),\\ {_2F_1}\left(\left.\begin{array}{c}
\frac{1}{3},\frac{2}{3} \\[4pt]
1 \\
\end{array}  \right|1-\frac{27 p^2 (1+p)^2}{4 (1+p+p^2)^3}\right)={}&\frac{2}{\pi}\frac{1+p+p^2}{\sqrt{3+6p}}\mathbf K\left( \sqrt{1-\frac{p^3(2+p)}{1+2p}} \right),\end{split}
\end{align}where $ p=\frac{2y}{1-y}$ and $y=\sqrt{\frac{1+u}{9+u}} $.

In what follows, we construct one more integral representation for $ W_4^S(1)$, by  Fourier analysis.\begin{proposition}[Parseval representation for $ W_4^{S}(1)$]We have the following formula:\begin{align}
W_{4}^{S}(1)=\frac{2}{\pi^{3}}\int_{0}^\infty\mathbf K\left( \frac{1}{1+ix} \right)\mathbf K\left( \frac{1}{1-ix} \right)\frac{\D x}{1+x^{2}}.
\end{align}\end{proposition}\begin{proof}
Following Glasser--Montaldi \cite[][(4), (5), (6b)]{GlasserMontaldi1993} and Zucker \cite[][(6.6)--(6.8)]{Zucker2011}, we deduce\begin{align}\begin{split}
W_{4}^{S}(1)={}&\frac{1}{\pi^4}\int_0^\infty\D t\int_{0}^\pi\D \phi_1
\int_{0}^\pi\D \phi_2\int_{0}^\pi\D \phi_3
\int_{0}^\pi\D \phi_{4}e^{-4t+t(\cos \phi_1+\cos \phi_2+\cos \phi_3+\cos \phi_4)}\\={}&\int_0^\infty e^{-4t}[I_0(t)]^4\D t.\end{split}
\end{align}By a special case of the Lipschitz--Hankel formula \cite[][\S13.22(2)]{Watson1944Bessel}, we have \begin{align}
\int_0^\infty e^{-2t}[I_0(t)]^2e^{-i\omega t  }\D t=\frac{1}{2+i\omega}\frac{2}{\pi}\mathbf K\left( \frac{2}{2+i\omega} \right),\quad \forall\omega\in(-\infty,0)\cup(0,\infty).\end{align}According to Parseval's theorem in Fourier analysis, we then obtain\begin{align}
W_{4}^{S}(1)=\frac{2}{\pi^{3}}\int_{-\infty}^\infty\mathbf K\left( \frac{2}{2+i\omega} \right)\mathbf K\left( \frac{2}{2-i\omega} \right)\frac{\D \omega}{4+\omega^{2}},
\end{align}which is equivalent to the claimed identity.\end{proof}Unfortunately, we have not found  straightforward hypergeometric transformations from any of the aforementioned single integrals to Laporta's representation in \eqref{eq:Laporta_pFq}. Therefore, we will use different methods for the proof of  Laporta's hypergeometric  sunrise formulae.

\subsection{Hypergeometric reduction of 4-loop sunrise}Now, we employ Mellin transforms and Meijer $G$-functions to prove  \eqref{eq:Laporta_pFq}. \begin{proposition}[Hypergeometric evaluation of 4-loop sunrise]We have the following identity:\begin{align}\begin{split}
\;\;\;\;\;
\dia{\put(-100,0){\line(1,0){200}}
\put(0,15){\circle{100}}
\put(0,-15){\circle{100}}
\put(50,0){\vtx}
\put(-50,0){\vtx}
}{}\;\;\;={}&2^{4}\int_0^\infty I_0(t)[K_0(t)]^5t\D t\\={}&\frac{4 \pi ^{5/2}}{\sqrt{3}}\left\{ \frac{\sqrt{3} }{2^6 }\left[\frac{\Gamma \left(\frac{1}{3}\right)}{\sqrt{\pi}}\right]^9\, _4F_3\left(\left. \begin{array}{c}\frac{1}{6},\frac{1}{3},\frac{1}{3},\frac{1}{2}\\[4pt]\frac{2}{3},\frac{5}{6},\frac{5}{6}\end{array} \right|1\right)-\frac{2^{4}}{3}\left[\frac{\sqrt{\pi}}{\Gamma \left(\frac{1}{3}\right)}\right]^9\, _4F_3\left(\left. \begin{array}{c}\frac{1}{2},\frac{2}{3},\frac{2}{3},\frac{5}{6}\\[4pt]\frac{7}{6},\frac{7}{6},\frac{4}{3}\end{array} \right|1\right) \right\}.\end{split}
\label{eq:IKM151_4F3_sum}\end{align} \end{proposition}\begin{proof}Combining \cite[][(3.1.11)]{Zhou2017WEF} with \cite[][(3.6)]{Rogers2009}, we put down\begin{align}\int_0^\infty I_{0}(xt)I_0(t)[K_0(t)]^3t\D t=
\frac{\pi^{2}}{4(4-x^{2})}
{_3F_2}\left(\left.\begin{array}{c}
\frac{1}{3},\frac{1}{2},\frac{2}{3} \\[4pt]1,1 \\
\end{array}\right|-\frac{108 x^2}{(4-x^2)^3}\right),\label{eq:IKM3F2}
\end{align}for $ x\in[0,2)$. We can rewrite the formula above by a contour integral representation of $ _3F_2$ \cite[][\S4.6.2]{Slater}:\begin{align}\begin{split}&
\int_0^\infty I_{0}(xt)I_0(t)[K_0(t)]^3t\D t\\={}&\frac{\pi^{2}}{4(4-x^{2})}\frac{1}{2\pi i}\int_{\delta-i\infty}^{\delta+i\infty}\frac{\sqrt{3}\Gamma \left(\frac{1}{3}-s\right) \Gamma \left(\frac{1}{2}-s\right) \Gamma \left(\frac{2}{3}-s\right)  \Gamma (s)}{2 \pi ^{3/2} [\Gamma (1-s)]^2}\left[\frac{108 x^2}{(4-x^2)^3}\right]^{-s}\D s,\end{split}\label{eq:IvKM231_MB}
\end{align} where $ \delta\in\left(0,\frac13\right)$.

By the Neumann   addition formula \cite[][\S11.2(1)]{Watson1944Bessel}, we have \begin{align}
[I_{0}(t)]^{2}=\frac{2}{\pi}\int_0^\pi I_0(2t\cos\theta)\D\theta,\label{eq:I0_add}
\end{align}so we can exploit Euler's beta integral to compute\begin{align}\begin{split}&
\int_0^\infty [I_{0}(t)]^{3} [K_0(t)]^3t\D t\\={}&\frac{1}{2\pi i}\int_{\delta-i\infty}^{\delta+i\infty}\frac{\Gamma \left(\frac{1}{3}-s\right) \left[\Gamma \left(\frac{1}{2}-s\right)\right]^2 \Gamma \left(\frac{2}{3}-s\right) \Gamma \left(s-\frac{1}{6}\right) \Gamma \left(s+\frac{1}{6}\right)}{32 \sqrt{3} \pi[  \Gamma (1-s)]^2}\D s\end{split}
\end{align}for $ \delta\in\left(\frac16,\frac13\right)$. According to the Whipple--Meijer formula \cite[][\S4.6.2]{Slater}, the right-hand side of the equation above evaluates to\begin{align}
\frac{\sqrt{3 \pi }}{4}\left\{ \frac{\sqrt{3} }{2^6 }\left[\frac{\Gamma \left(\frac{1}{3}\right)}{\sqrt{\pi}}\right]^9\, _4F_3\left(\left. \begin{array}{c}\frac{1}{6},\frac{1}{3},\frac{1}{3},\frac{1}{2}\\[4pt]\frac{2}{3},\frac{5}{6},\frac{5}{6}\end{array} \right|1\right)-\frac{2^{4}}{3}\left[\frac{\sqrt{\pi}}{\Gamma \left(\frac{1}{3}\right)}\right]^9\, _4F_3\left(\left. \begin{array}{c}\frac{1}{2},\frac{2}{3},\frac{2}{3},\frac{5}{6}\\[4pt]\frac{7}{6},\frac{7}{6},\frac{4}{3}\end{array} \right|1\right)\right\},
\end{align} which is also the same as  $ \frac{3}{\pi^{2}}\int_0^\infty I_0(t)[K_0(t)]^5t\D t$.    \end{proof}\begin{proposition}[Laporta's single integral for 4-loop sunrise]\label{prop:Laporta_Pnu_sqr}We have \begin{align}\begin{split}{}&
\int_{0}^1\left[ _2F_1\left(\left.\begin{array}{c}
\frac{1}{3},\frac{2}{3} \\
1 \\
\end{array}  \right|x\right)\right]^2\frac{\D x}{\sqrt{1-x}}\\={}&G_{4,4}^{2,2}\left(1\left|
\begin{array}{c}
 \frac{1}{2},\frac{1}{2},\frac{1}{3},\frac{2}{3} \\[4pt]
 0,0,-\frac{1}{6},\frac{1}{6} \\
\end{array}
\right.\right)=\frac{3}{4 \pi ^2}G_{4,4}^{2,4}\left(1\left|
\begin{array}{c}
 \frac{1}{3},\frac{1}{2},\frac{1}{2},\frac{2}{3} \\[4pt]
 0,0,-\frac{1}{6},\frac{1}{6} \\
\end{array}
\right.\right)\\={}&\frac{9}{\sqrt{\pi }}\left\{ \frac{\sqrt{3} }{2^6 }\left[\frac{\Gamma \left(\frac{1}{3}\right)}{\sqrt{\pi}}\right]^9\, _4F_3\left(\left. \begin{array}{c}\frac{1}{6},\frac{1}{3},\frac{1}{3},\frac{1}{2}\\[4pt]\frac{2}{3},\frac{5}{6},\frac{5}{6}\end{array} \right|1\right)-\frac{2^{4}}{3}\left[\frac{\sqrt{\pi}}{\Gamma \left(\frac{1}{3}\right)}\right]^9\, _4F_3\left(\left. \begin{array}{c}\frac{1}{2},\frac{2}{3},\frac{2}{3},\frac{5}{6}\\[4pt]\frac{7}{6},\frac{7}{6},\frac{4}{3}\end{array} \right|1\right)\right\}.\end{split}\label{eq:LaportaMeijer}
\end{align}\end{proposition}
\begin{proof}First, we transcribe  \cite[][p.~316, (15)]{ET2} as follows:\begin{align}
\int_0^1 {_2F_1}\left(\left.\begin{array}{c}
-\nu,\nu+1\ \\[4pt]
1 \\
\end{array}  \right|1-t\right )t^{s-1}\D t=\frac{[\Gamma (s)]^2}{\Gamma (s-\nu ) \Gamma (s+\nu +1)},\quad \R s>0.\label{eq:Pnu_Mellin}
\end{align}By Mellin convolution, we have \begin{align}\begin{split}&
\int_0^1 \left[{_2F_1}\left(\left.\begin{array}{c}
-\nu,\nu+1\ \\[4pt]
1 \\
\end{array}  \right|1-t\right )\right]^2t^\alpha\D t\\={}&\frac{1}{2\pi i}\int_{\delta-i\infty}^{\delta+i\infty}\frac{[\Gamma (\alpha +1-s)]^2[\Gamma (s)]^2\D s}{\Gamma (s-\nu ) \Gamma (s+\nu +1)\Gamma (\alpha +1- s-\nu ) \Gamma (\alpha +2-s+\nu)},\end{split}
\end{align}where $ \alpha\in(-1,\infty),\delta\in(0,\alpha+1)$. Setting $ \alpha=-1/2,\nu=-1/3$ in the equation above, we can verify the first equality in \eqref{eq:LaportaMeijer}.

Before proving the second equality in \eqref{eq:LaportaMeijer}, we note that the Meijer $G$-function\begin{align}
G_{4,4}^{2,2}\left(z\left|
\begin{array}{c}
 \frac{1}{2},\frac{1}{2},\frac{1}{3},\frac{2}{3} \\[4pt]
 0,0,-\frac{1}{6},\frac{1}{6} \\
\end{array}
\right.\right)
\end{align}is annihilated by a fourth-order differential operator \cite[][(34)]{Meijer1946b}:\begin{align}
z\left( z\frac{\D}{\D z}+\frac{1}{3} \right)\left( z\frac{\D}{\D z}+\frac{1}{2} \right)^2\left( z\frac{\D}{\D z}+\frac{2}{3} \right)-\left(z\frac{\D}{\D z}\right)^2\left( z\frac{\D}{\D z}-\frac{1}{6} \right)\left( z\frac{\D}{\D z}+\frac{1}{6} \right).
\end{align}More generally, the kernel space of this  differential operator is spanned by four functions:\begin{align}\left\{\begin{array}{r@{\,=\,}l}f_{1}(z)&\smash[t]{\dfrac{1}{z^{1/6}}{_4F_3}\left(\left. \begin{array}{c}\frac{1}{6},\frac{1}{3},\frac{1}{3},\frac{1}{2}\\[4pt]\frac{2}{3},\frac{5}{6},\frac{5}{6}\end{array} \right|z\right)},\\[12pt]f_2(z)&z^{1/6}{_4F_3}\left(\left. \begin{array}{c}\frac{1}{2},\frac{2}{3},\frac{2}{3},\frac{5}{6}\\[4pt]\frac{7}{6},\frac{7}{6},\frac{4}{3}\end{array} \right|z\right) ,\\[12pt]f_3(z)&{_4F_3}\left(\left. \begin{array}{c}\frac{1}{3},\frac{1}{2},\frac{1}{2},\frac{2}{3}\\[4pt]\frac{5}{6},1,\frac{7}{6}\end{array} \right|z\right) ,\\[12pt]f_4(z)&\smash[b]{G_{4,4}^{2,4}\left(z\left|
\begin{array}{c}
 \frac{1}{3},\frac{1}{2},\frac{1}{2},\frac{2}{3} \\[4pt]
 0,0,-\frac{1}{6},\frac{1}{6} \\
\end{array}
\right.\right)}.\end{array}\right.
\end{align}  which  exhibit the following asymptotic behavior, as $z\to0^+ $:\begin{align}\left\{\begin{array}{r@{\,=\,}l}f_{1}(z)&\smash[t]{\dfrac{1}{z^{1/6}}+\dfrac{z^{5/6}}{50}+O(z^{11/6})},\\[12pt]f_2(z)&z^{1/6}+\dfrac{5 z^{7/6}}{49}+O(z^{13/6}),\\[12pt]f_3(z)&1+\dfrac{2 z}{35}+O(z^2),\\[12pt]f_4(z)&\smash[b]{2 \sqrt{3} \pi  (6-\log z)+\dfrac{2 \sqrt{3} \pi  z (109-70 \log z)}{1225}+O(z^2\log z)}.\end{array}\right.
\end{align}Comparing the list above  with \begin{align}
G_{4,4}^{2,2}\left(z\left|
\begin{array}{c}
 \frac{1}{2},\frac{1}{2},\frac{1}{3},\frac{2}{3} \\[4pt]
 0,0,-\frac{1}{6},\frac{1}{6} \\
\end{array}
\right.\right)=\frac{3 \sqrt{3} (6-\log z)}{2 \pi }+\frac{3 \sqrt{3} z (109-70 \log z)}{2450 \pi }+O(z^2\log z),
\end{align}we can show that\begin{align}
G_{4,4}^{2,2}\left(z\left|
\begin{array}{c}
 \frac{1}{2},\frac{1}{2},\frac{1}{3},\frac{2}{3} \\[4pt]
 0,0,-\frac{1}{6},\frac{1}{6} \\
\end{array}
\right.\right)=\frac{3}{4 \pi ^2}G_{4,4}^{2,4}\left(z\left|
\begin{array}{c}
 \frac{1}{3},\frac{1}{2},\frac{1}{2},\frac{2}{3} \\[4pt]
 0,0,-\frac{1}{6},\frac{1}{6} \\
\end{array}
\right.\right),
\end{align} which embodies the second equality in \eqref{eq:LaportaMeijer} as a special case.

To prove the last equality in \eqref{eq:LaportaMeijer}, we apply residue calculus to the Mellin--Barnes integral  representation of the Meijer $G$-function in question. Concretely speaking,  by closing the contour rightwards, we have \begin{align}\begin{split}&
\frac{3}{4 \pi ^2}G_{4,4}^{2,4}\left(1\left|
\begin{array}{c}
 \frac{1}{3},\frac{1}{2},\frac{1}{2},\frac{2}{3} \\[4pt]
 0,0,-\frac{1}{6},\frac{1}{6} \\
\end{array}
\right.\right)=\frac{3}{4 \pi ^2}\frac{1}{2\pi i}\int_{\frac14-i\infty}^{\frac14+i\infty}\frac{\Gamma \left(\frac{1}{3}-s\right) \left[\Gamma \left(\frac{1}{2}-s\right)\right]^2 \Gamma \left(\frac{2}{3}-s\right) [\Gamma (s)]^2}{\Gamma \left(\frac{5}{6}-s\right) \Gamma \left(\frac{7}{6}-s\right)}\D s\\={}&\frac{3}{4 \pi ^2}\sum_{n=0}^\infty\frac{(-1)^{n} \left[\Gamma \left(\frac{1}{6}-n\right)\right]^2 \Gamma \left(\frac{1}{3}-n\right)\left[ \Gamma \left(n+\frac{1}{3}\right)\right]^2}{n! \Gamma \left(\frac{1}{2}-n\right) \Gamma \left(\frac{5}{6}-n\right)}+\frac{3}{4 \pi ^2}\sum_{n=0}^{\infty}\frac{\Gamma \left(-\frac{1}{6}-n\right) \Gamma \left(\frac{1}{6}-n\right) \left[\Gamma \left(n+\frac{1}{2}\right)\right]^2 }{(n!)^2 \Gamma \left(\frac{1}{3}-n\right) \Gamma \left(\frac{2}{3}-n\right)}\times\\{}&\times\left[-\psi ^{(0)}\left(\frac{1}{3}-n\right)-\psi ^{(0)}\left(\frac{2}{3}-n\right)+\psi ^{(0)}\left(\frac{1}{6}-n\right)+\psi ^{(0)}\left(-\frac{1}{6}-n\right)+2 \psi ^{(0)}(n+1)-2 \psi ^{(0)}\left(n+\frac{1}{2}\right)\right]\\{}&+\frac{3}{4 \pi ^2}\sum_{n=0}^\infty\frac{(-1)^{n} \Gamma \left(-\frac{1}{3}-n\right) \left[\Gamma \left(-\frac{1}{6}-n\right)\right]^2 \left[\Gamma \left(n+\frac{2}{3}\right)\right]^2}{n!\Gamma \left(\frac{1}{6}-n\right) \Gamma \left(\frac{1}{2}-n\right)},
\end{split}\end{align}  where the three sums are attributed to residues at $ s=n+\frac13,n+\frac{1}{2},n+\frac{2}{3}$ for $ n\in\mathbb Z_{\geq0}$, and $ \psi^{(0)}(z)=\D \log \Gamma(z)/\D z$; by closing the contour leftwards and collecting residues at  $ s=-n$ for $ n\in\mathbb Z_{\geq0}$, we get \begin{align}\begin{split}&
\frac{3}{4 \pi ^2}G_{4,4}^{2,4}\left(1\left|
\begin{array}{c}
 \frac{1}{3},\frac{1}{2},\frac{1}{2},\frac{2}{3} \\[4pt]
 0,0,-\frac{1}{6},\frac{1}{6} \\
\end{array}
\right.\right)=\frac{3}{4 \pi ^2}\frac{1}{2\pi i}\int_{\frac14-i\infty}^{\frac14+i\infty}\frac{\Gamma \left(\frac{1}{3}-s\right) \left[\Gamma \left(\frac{1}{2}-s\right)\right]^2 \Gamma \left(\frac{2}{3}-s\right) [\Gamma (s)]^2}{\Gamma \left(\frac{5}{6}-s\right) \Gamma \left(\frac{7}{6}-s\right)}\D s\\={}&-\frac{1}{4 \pi ^2}\sum_{n=0}^{\infty}\frac{\Gamma \left(-\frac{1}{6}-n\right) \Gamma \left(\frac{1}{6}-n\right) \left[\Gamma \left(n+\frac{1}{2}\right)\right]^2 }{(n!)^2 \Gamma \left(\frac{1}{3}-n\right) \Gamma \left(\frac{2}{3}-n\right)}\times
\\{}&\times\left[-\psi ^{(0)}\left(\frac{1}{3}-n\right)-\psi ^{(0)}\left(\frac{2}{3}-n\right)+\psi ^{(0)}\left(\frac{1}{6}-n\right)+\psi ^{(0)}\left(-\frac{1}{6}-n\right)+2 \psi ^{(0)}(n+1)-2 \psi ^{(0)}\left(n+\frac{1}{2}\right)\right].\end{split}\end{align} Eliminating the last series from the last pair of equations, we obtain\begin{align}\begin{split}&
\frac{3}{4 \pi ^2}G_{4,4}^{2,4}\left(1\left|
\begin{array}{c}
 \frac{1}{3},\frac{1}{2},\frac{1}{2},\frac{2}{3} \\[4pt]
 0,0,-\frac{1}{6},\frac{1}{6} \\
\end{array}
\right.\right)\\={}&\frac{3}{16 \pi ^2}\left\{\sum_{n=0}^\infty\frac{(-1)^{n} \left[\Gamma \left(\frac{1}{6}-n\right)\right]^2 \Gamma \left(\frac{1}{3}-n\right)\left[ \Gamma \left(n+\frac{1}{3}\right)\right]^2}{n! \Gamma \left(\frac{1}{2}-n\right) \Gamma \left(\frac{5}{6}-n\right)}+\sum_{n=0}^\infty\frac{(-1)^{n} \Gamma \left(-\frac{1}{3}-n\right) \left[\Gamma \left(-\frac{1}{6}-n\right)\right]^2 \left[\Gamma \left(n+\frac{2}{3}\right)\right]^2}{n!\Gamma \left(\frac{1}{6}-n\right) \Gamma \left(\frac{1}{2}-n\right)}\right\}\\={}&\frac{9}{\sqrt{\pi }}\left\{ \frac{\sqrt{3} }{2^6 }\left[\frac{\Gamma \left(\frac{1}{3}\right)}{\sqrt{\pi}}\right]^9\, _4F_3\left(\left. \begin{array}{c}\frac{1}{6},\frac{1}{3},\frac{1}{3},\frac{1}{2}\\[4pt]\frac{2}{3},\frac{5}{6},\frac{5}{6}\end{array} \right|1\right)-\frac{2^{4}}{3}\left[\frac{\sqrt{\pi}}{\Gamma \left(\frac{1}{3}\right)}\right]^9\, _4F_3\left(\left. \begin{array}{c}\frac{1}{2},\frac{2}{3},\frac{2}{3},\frac{5}{6}\\[4pt]\frac{7}{6},\frac{7}{6},\frac{4}{3}\end{array} \right|1\right)\right\}, \end{split}
\end{align}by direct summation.\end{proof}\begin{remark}If all the poles in  $ \prod_{j=1}^n\Gamma(1-a_{j}-s)$ [resp.~$\prod_{k=1}^m\Gamma(b_{k}+s)$] are simple in  \eqref{eq:MeijerG_defn}, then the Meijer $G$-function $ G^{m,n}_{p,q}$ decomposes into a linear combination of $ _qF_{p-1}$ (resp.~$_pF_{q-1}$), as indicated in \cite[][5.3(6)]{HTF1} (resp.~\cite[][5.3(5)]{HTF1}). Such a standard decomposition does not apply to the two  $G$-functions in \eqref{eq:LaportaMeijer}.   \eor\end{remark}\section{Analogs of Laporta's 4-loop sunrise formulae\label{sec:sunrise_analogs}}\subsection{Bailey--Meijer reductions of certain hypergeometric series}In the notations of Zudilin \cite[][Proposition 2]{Zudilin2004odd_zeta} and Borwein--Straub--Wan \cite[][Figure 3]{BSW2013},  we paraphrase an identity of Bailey \cite[][(3.4)]{Bailey1932} in terms of the Meijer $G$-function:\begin{align}
\begin{split}&
{_7F_6}\left(\left. \begin{array}{c}
a,1+\frac{a}{2},b,c,d,e,f \\[4pt]
\frac{a}{2},1+a-b,1+a-c,1+a-d,1+a-e,1+a-f \\
\end{array} \right|1\right)\\={}&\frac{\Gamma (1+a-b) \Gamma (1+a-c) \Gamma (1+a-d) \Gamma (1+a-e) \Gamma (1+a-f)}{\Gamma (a+1) \Gamma (b) \Gamma (c) \Gamma (d) \Gamma (1+a-b-c) \Gamma (1+a-b-d) \Gamma (1+a-c-d) \Gamma (1+a-e-f)}\times\\{}&\times G_{4,4}^{2,4}\left(1\left|
\begin{array}{c}
 e+f-a,1-b,1-c,1-d \\[4pt]
 0,1+a-b-c-d,e-a,f-a \\
\end{array}
\right.\right).
\end{split}   \label{eq:Bailey7F6}
\end{align}
\begin{proposition}[Bailey representations of 4-loop sunrise]We have \begin{align}\begin{split}
\int_{0}^1\left[ _2F_1\left(\left.\begin{array}{c}
\frac{1}{3},\frac{2}{3} \\[4pt]
1 \\
\end{array}  \right|x\right)\right]^2\frac{\D x}{\sqrt{1-x}}
={}&\frac{9}{2} \, _7F_6\left(\left.   \begin{array}{c}
\frac{1}{3},\frac{1}{3},\frac{1}{2},\frac{1}{2},\frac{2}{3},\frac{2}{3},\frac{5}{4} \\[4pt]
\frac{1}{4},\frac{5}{6},\frac{5}{6},1,\frac{7}{6},\frac{7}{6} \\
\end{array}\right|1\right)\\={}&2^{14/3} \sqrt{3} \left[\frac{\sqrt{\pi }}{\Gamma \left(\frac{1}{3}\right)}\right]^6{_6F_5}\left(\left.   \begin{array}{c}
\frac{1}{2},\frac{1}{2},\frac{1}{2},\frac{2}{3},\frac{2}{3},\frac{4}{3} \\[4pt]
\frac{1}{3},1,\frac{7}{6},\frac{7}{6},\frac{7}{6} \\
\end{array}\right|1\right)\\={}&\frac{3 \sqrt{3} }{2^{11/3}}\left[\frac{\Gamma \left(\frac{1}{3}\right)}{\sqrt{\pi }}\right]^6{_5F_4}\left(\left.   \begin{array}{c}
\frac{1}{3},\frac{1}{3},\frac{1}{2},\frac{1}{2},\frac{1}{2} \\[4pt]
\frac{5}{6},\frac{5}{6},\frac{5}{6},1 \\
\end{array}\right|1\right).\end{split}\label{eq:Bailey_redn}\end{align}\end{proposition}\begin{proof}There are 24 different choices of $a,b,c,d,e,f$ that make Bailey's identity applicable to the special $ G_{4,4}^{2,4}$ appearing in \eqref{eq:LaportaMeijer}. Due to the invariance of the generalized hypergeometric series
\begin{align}{_7F_6}\left(\left.\begin{array}{c}
a_{1},\dots,a_7 \\[4pt]
b_{1},\dots,b_6 \\
\end{array}\right| 1\right):=1+\sum_{n=1}^\infty\frac{\prod_{j=1}^7(a_{j})_{n}}{\prod_{k=1}^6(b_k)_{n} }\frac{1}{n!},
\end{align}  under permutations of its parameters, we are left with only three distinct forms of $_7F_6$ as outputs from Bailey's identity. One of them simplifies to $ _6F_5$ (resp.~$ _5F_4$), with cancelations from $ a_1=b_1$ (resp.~$ a_1=b_1,a_2=b_2$). This explains all the stated results.    \end{proof}
\begin{remark}
Following Wan \cite[][Theorem 1]{Wan2012}, we recapitulate a special case of Zudilin's integral formula \cite{Zudilin2002}:\begin{align}\begin{split}
 &  \int_0^1 \int_0^1 \int_0^1 \frac{x^{a_2-1}y^{a_3-1}z^{a_4-1}(1-x)^{a_0-a_2-a_3} (1-y)^{a_0-a_3-a_4} (1-z)^{a_0-a_4-a_5}}{\{1-x[1-y(1-z)]\}^{a_1}} \, \D x\D y\D z\\
 = {}& \frac{\Gamma(a_0+1)\prod_{j=2}^4 \Gamma(a_j) \prod_{j=1}^4 \Gamma(a_0+1-a_j-a_{j+1})}{\prod_{j=1}^5 \Gamma(a_0+1-a_j)}\times  \\
 {}&\times{ _7F_6}\left(\left. {{a_0,1+\frac{a_{0}}{2},a_1,a_2,a_3,a_4,a_5}\atop{\frac{a_0}{2},1+a_0-a_1,1+a_0-a_2,1+a_0-a_3,1+a_0-a_4,1+a_0-a_5}}\right| 1 \right),\end{split}\label{eq:Zudilin7F6}
\end{align}where the  chosen indices $ a_0,a_1,\dots,a_5$ ensure convergence of both sides.
Using \begin{align}
\int_\xi^1\frac{\D y}{\sqrt[3]{y(1-y)(y-\xi)^2}}={}&\frac{2\pi}{\sqrt{3}} {_2F_1}\left(\left.\begin{array}{c}
\frac{1}{3},\frac{2}{3} \\[4pt]
1 \\
\end{array}  \right|1-\xi\right),\\\int_0^1\frac{\D x}{\sqrt[3]{x(1-x)^{2}\{1-x[1-y(1-z)]\}}}={}&\frac{2\pi}{\sqrt{3}} {_2F_1}\left(\left.\begin{array}{c}
\frac{1}{3},\frac{2}{3} \\[4pt]
1 \\
\end{array}  \right|1-y(1-z)\right),
\end{align}and a variation on Wan's method \cite[][p.~124]{Wan2012}, we can show that  \begin{align}\begin{split}
\int_{0}^1\left[ _2F_1\left(\left.\begin{array}{c}
\frac{1}{3},\frac{2}{3} \\
1 \\
\end{array}  \right|1-\xi\right)\right]^2\frac{\D \xi}{\sqrt{\xi}}={}&\frac{\sqrt{3}}{2\pi}\int_{0}^1\left[\int_\xi^1\frac{\D y}{\sqrt[3]{\vphantom1\smash[b]y(1-\vphantom1\smash[b]y)(\vphantom1\smash[b]y-\smash[b]{\xi})^2}}\right] {_2F_1}\left(\left.\begin{array}{c}
\frac{1}{3},\frac{2}{3} \\[4pt]
1 \\
\end{array}  \right|1-\xi\right)\frac{\D \xi}{\sqrt{\xi}}\\\xlongequal{\xi= y(1-z)}{}&\frac{\sqrt{3}}{2\pi}\int_{0}^1\int_0^1\frac{{_2F_1}\left(\left.\begin{smallmatrix}
\frac{1}{3},\frac{2}{3} \\
1 \\
\end{smallmatrix}  \right|1-y(1-z)\right)}{\sqrt{\vphantom1\smash[b]y}\sqrt[3]{(1-y)z^2}} \frac{\D y\D z}{\sqrt{1-z}}\\={}&\frac{3}{4\pi^2}\int_0^1 \int_0^1 \int_0^1 \sqrt[3]{\frac{(1-x)^{-2} (1-y)^{-1} }{{1-x[1-y(1-z)]}}} \, \frac{\D x\D y\D z}{\sqrt{ \vphantom1\smash[b]y(1-z)}\sqrt[3]{xz^{2}}}.
\end{split}\end{align}Setting $ a_0=\frac{1}{2},a_1=\frac{1}{3},a_2=\frac{2}{3},a_3=\frac{1}{2},a_4=\frac{1}{3},a_5=\frac{2}{3}$ in  \eqref{eq:Zudilin7F6}, we immediately recover the first equality in \eqref{eq:Bailey_redn} from Zudilin's formula.\eor\end{remark}\begin{remark} At present,  aside from  [cf.~\eqref{eq:spec_L_val}]\begin{align}\begin{split}&
\frac{4  \sqrt{3} \pi ^3}{27} \int_{0}^1\left[ _2F_1\left(\left.\begin{array}{c}
\frac{1}{3},\frac{2}{3} \\[4pt]
1 \\
\end{array}  \right|x\right)\right]^2\frac{\D x}{\sqrt{1-x}}\\={}&8\pi ^{2}L(f_{4,6},2)=-32\pi^4\int_{0}^{i\infty}[\eta(z)\eta(2z)\eta(3z)\eta(6z)]^{2}z\D z,\end{split}
\end{align} we are not able to further reduce the special values of $ _pF_{p-1}$ (with $ p\in\{7,6,5,4\}$) appearing in  \eqref{eq:LaportaMeijer} and  \eqref{eq:Bailey_redn}  to more familiar mathematical constants.  However, we do not exclude the possibility of finding their closed-form evaluations\footnote{It is arguable whether $ L(f_{4,6},2)$ should count as a closed-form evaluation in its own right. As one may recall, Bloch--Kerr--Vanhove \cite{BlochKerrVanhove2015} and Samart \cite{Samart2016}  have expressed the 3-loop sunrise diagram $ 2^3\int_0^\infty I_0(t)[K_0(t)]^4t\D t$  as $ \frac{12\pi}{\sqrt{15}}L(f_{3,15},2)$, for a modular form $f_{3,15}(z)=[\eta(3z)\eta(5z)]^3+[\eta(z)\eta(15z)]^3$ of weight 3 and  level 15. Meanwhile, according to the work of Rogers--Wan--Zucker \cite{RogersWanZucker2015}, such a special $L$-value can be reduced to a product of  gamma values at rational arguments, thus leaving us a formula $ 2^3\int_0^\infty I_0(t)[K_0(t)]^4t\D t=\frac{1}{30\sqrt{5}}\Gamma\left( \frac{1}{15} \right)\Gamma\left( \frac{2}{15} \right)\Gamma\left( \frac{4}{15} \right)\Gamma\left( \frac{8}{15} \right)$ (see \cite[][Theorem 2.2.2]{Zhou2017WEF} for a simplified proof of this integral identity). At the time of writing, it is not clear to us if the special $L$-value  $ L(f_{4,6},2)$ admits a similar reduction.} in future efforts. Later in this article, we will also keep some generalized hypergeometric expressions ``unevaluated'', due to our current lack of quantitative understanding for them.   \eor\end{remark}

\subsection{Mellin--Barnes representations of some Bessel moments}To prepare for computations later in this article, we represent certain linear combinations of Feynman diagrams as Meijer $G$-functions.
\begin{lemma}[Mellin--Barnes representations for  Feynman integrals]For $u\in(0,4)$, we have \begin{align}
\begin{split}&\int_0^\infty I_0(\sqrt{u}t)[K_0(t)]^4t\D t+4\int_0^\infty K_0(\sqrt{u}t)I_{0}(t)[K_0(t)]^3t\D t\\={}&\frac{\sqrt{3}\pi^{3/2}}{4(4-u)}\frac{1}{2\pi i }\int_{\frac14-i\infty}^{\frac14+i\infty}\frac{ \Gamma \left(\frac{1}{3}-s\right) \Gamma \left(\frac{2}{3}-s\right) \left[\Gamma (s)\right]^2}{ \Gamma (1-s) \Gamma \left(s+\frac{1}{2}\right)}\left[ \frac{108 u}{(4-u)^3} \right]^{-s}\D s\label{eq:MB1}\end{split}
\intertext{and}\begin{split}&\int_{0}^\infty K_{0}(\sqrt{u}t)[I_{0}(t)]^{2} [K_0(t)]^2t\D t+\int_{0}^\infty I_{0}(\sqrt{u}t)I_{0}(t) [K_0(t)]^3t\D t\\={}&\frac{\sqrt{3}}{8 \pi ^{3/2} (4-u)}\frac{1}{2\pi i }\int_{\frac14-i\infty}^{\frac14+i\infty}\Gamma \left(\frac{1}{3}-s\right) \Gamma \left(\frac{1}{2}-s\right) \Gamma \left(\frac{2}{3}-s\right) [\Gamma (s)]^3\left[ \frac{108 u}{(4-u)^3} \right]^{-s}\D s.\label{eq:MB2}\end{split}\end{align}\end{lemma}
\begin{proof}From \cite[][Lemma 4.2]{Zhou2017BMdet}, we know that the left-hand sides of both \eqref{eq:MB1} and \eqref{eq:MB2} are annihilated by Vanhove's third-order differential operator  \cite[][Table 1, $n=4$]{Vanhove2014Survey}\begin{align}\begin{split}{}&
u^2 (u - 4) (u - 16)\frac{\D^3}{\D u^3}+6 u (u^2 - 15 u + 32)\frac{\D^2}{\D u^2}\\&+(7 u^2 - 68 u + 64)\frac{\D}{\D u}+(u-4).\label{eq:VL3}
\end{split}\end{align}Suppose that  the left-hand of either  \eqref{eq:MB1} or \eqref{eq:MB2} takes the form $ \frac{1}{4-u}g\left( -\frac{108 u}{(4-u)^3} \right)$, and set  $ w=-\frac{108 u}{(4-u)^3}$, then we can check that $ g(w)$ satisfies the following homogeneous differential equation:\begin{align}
9 (w-1) w^2 g'''(w)+\frac{27}{2} (3 w-2) w g''(w)+(29 w-9) g'(w)+g(w)=0.
\end{align}Thus, the function $ g(w)$ must be a linear combination of three solutions:\vspace{.5em}\begin{align}
\left\{\begin{array}{r@{\,=\,}l}g_1(w)&\smash[t]{{_3F_2}\left(\left.\begin{array}{c}
\frac{1}{3},\frac{1}{2},\frac{2}{3} \\[4pt]1,1 \\
\end{array}\right|w\right)},\\[12pt]g_2(w)&G_{3,3}^{2,3}\left(w\left|
\begin{array}{c}
 \frac{1}{3},\frac{1}{2},\frac{2}{3} \\[4pt]
 0,0,0 \\
\end{array}
\right.\right),\\[12pt]g_3(w)&\smash[b]{G_{3,3}^{3,3}\left(-w\left|
\begin{array}{c}
 \frac{1}{3},\frac{1}{2},\frac{2}{3} \\[4pt]
 0,0,0 \\
\end{array}
\right.\right)}.\end{array}\right.
\end{align}\vspace{.3em}

\noindent The exact contribution from each member in this basis set can be determined by asymptotic analysis, which will occupy the rest of this proof.

First, we consider   \eqref{eq:MB1}. In \cite[][Propositions 3.1.2 and 5.1.4]{Zhou2017WEF}, we have effectively shown  that \begin{align}\begin{split}&
\int_0^\infty I_0(\sqrt{u}t)[K_0(t)]^4t\D t+4\int_0^\infty K_0(\sqrt{u}t)I_{0}(t)[K_0(t)]^3t\D t\\={}&\frac{\pi^4}{6}\int_0^\infty J_0(\sqrt{u}t)[J_0(t)]^4t\D t=:\frac{\pi^{4}}{6}\frac{p_4(\sqrt{u})}{\sqrt{u}}\end{split}\label{eq:p4_norm}
\end{align}holds for  $ 0<u<4$. Here, $ J_0(x):=\frac{2}{\pi}\int_0^{\pi/2}\cos(x\cos\varphi)\D \varphi$ is the Bessel function of the first kind and zeroth order, while $ p_4(x):=\int_0^\infty J_0(xt)[J_0(t)]^4 xt\D t,x>0$ is Kluyver's probability density for the distance $x$ traveled by a rambler walking in the Euclidean plane, taking 4 consecutive and independent unit steps, each aiming at uniformly distributed directions \cite{BSWZ2012}. As $ u\to 0^+$, we compare\begin{align}
\left\{\begin{array}{r@{\,=\,}l}\dfrac{1}{4-u}g_1\left(-\dfrac{108 u}{(4-u)^3}\right)&\dfrac{1}{4}+O(u),\\[12pt]\dfrac{1}{4-u}g_2\left(-\dfrac{108 u}{(4-u)^3}\right)&\dfrac{\pi ^{3/2} }{2 \sqrt{3}}\left(-i \pi +\log\dfrac{64}{u} \right)+O(u\log u),\\[12pt]\dfrac{1}{4-u}g_3\left(-\dfrac{108 u}{(4-u)^3}\right)&\dfrac{\pi ^{3/2}\log u }{4 \sqrt{3}}\log \dfrac{u}{4096}+O(1) ,\end{array}\right.\label{eq:near0}
\end{align} with the asymptotic behavior of Kluyver's probability density \cite[][Example 4.3 and Theorem 4.4]{BSWZ2012} \begin{align} \frac{p_4(\sqrt{u})}{\sqrt u}=-\frac{3\log u}{4\pi^{2}}+O(1),\end{align}  we  arrive at an expression \begin{align}\begin{split}
\frac{p_4(\sqrt{u})}{\sqrt u}={}&\frac{3 \sqrt{3}}{2 \pi ^{7/2}(4-u)}\left[ \frac{2 i \pi ^{5/2}}{\sqrt{3}}g_1\left( -\frac{108 u}{(4-u)^3} \right)+g_{2} \left( -\frac{108 u}{(4-u)^3} \right)\right]\\={}&\frac{1}{4-u}\frac{1}{2\pi i }\int_{\frac14-i\infty}^{\frac14+i\infty}\frac{3 \sqrt{3} \Gamma \left(\frac{1}{3}-s\right) \Gamma \left(\frac{2}{3}-s\right) \left[\Gamma (s)\right]^2}{2 \pi ^{5/2} \Gamma (1-s) \Gamma \left(s+\frac{1}{2}\right)}\left[ \frac{108 u}{(4-u)^3} \right]^{-s}\D s.\end{split}\label{eq:p4_Mellin_Barnes}
\end{align} This proves    \eqref{eq:MB1}.

Next, we study   \eqref{eq:MB2}, which essentially says that \begin{align}\begin{split}&
\int_{0}^\infty K_{0}(\sqrt{u}t)[I_{0}(t)]^{2} [K_0(t)]^2t\D t\\={}&\frac{\sqrt{3}}{8 \pi ^{3/2} (4-u)}g_3\left(-\dfrac{108 u}{(4-u)^3}\right)-\frac{\pi^2}{4(4-u)}g_1\left(-\dfrac{108 u}{(4-u)^3}\right).\end{split}\tag{\ref{eq:MB2}$'$}\label{eq:MB2'}
\end{align} We need two stages of asymptotic analysis to verify the identity above, which will described in the two paragraphs to follow.

As $u\to0^+$, we have \cite[cf.][Proposition 2.5]{Zhou2017BMdet}\begin{align}\begin{split}&
\int_{0}^\infty K_{0}(\sqrt{u}t)[I_{0}(t)]^{2} [K_0(t)]^2t\D t\\={}&\frac12\int_{0}^\infty K_{0}(\sqrt{u}t)I_{0}(t)K_0(t)\D t\\{}&+\int_{0}^\infty K_{0}(\sqrt{u}t)I_{0}(t)K_0(t)\left[ I_{0}(t)K_0(t)-\frac{1}{2t}\right]t\D t\\={}&\frac12\int_{0}^\infty K_{0}(\sqrt{u}t)I_{0}(t)K_0(t)\D t+O(\log u),\end{split}
\end{align}    where Bailey's integral formula \cite[cf.][(3.3)]{Bailey1936II} leads us to \begin{align}\begin{split}&\frac
12\int_{0}^\infty K_{0}(\sqrt{u}t)I_{0}(t)K_0(t)\D t\\={}&\frac{1}{2\sqrt{u}}\mathbf K\left( \sqrt{\frac{1-i\sqrt{(4-u)/u}}{2}} \right)\mathbf K\left( \sqrt{\frac{1+i\sqrt{(4-u)/u}}{2}}\right)\\={}&\frac{1}{32}\log^2\frac{4}{u}+O(\log u). \end{split}
\end{align}  So far, we know that [cf.~the last line in \eqref{eq:near0},  and the first equality in \eqref{eq:p4_Mellin_Barnes}]\begin{align}\begin{split}&
\int_{0}^\infty K_{0}(\sqrt{u}t)[I_{0}(t)]^{2} [K_0(t)]^2t\D t\\={}&\frac{\sqrt{3}}{8 \pi ^{3/2} (4-u)}g_3\left(-\dfrac{108 u}{(4-u)^3}\right)+\frac{A}{4-u}g_1\left(-\dfrac{108 u}{(4-u)^3}\right)+\frac{Bp_4(\sqrt{u})}{\sqrt u}\end{split}\label{eq:constAB}
\end{align}    for certain constants $A$ and $B$.

In the regime where $u\to4^-$, we have\begin{align}
\left\{\begin{array}{r@{\,=\,}l}\dfrac{1}{4-u}g_1\left(-\dfrac{108 u}{(4-u)^3}\right)&\dfrac{3 }{2^{14/3} \pi }\left[\dfrac{\Gamma \left(\frac{1}{3}\right)}{\sqrt{\pi }}\right]^6-\dfrac{\sqrt{4-u}}{2 \pi }+O(4-u),\\[12pt]\dfrac{p_4(\sqrt{u})}{\sqrt u}&\dfrac{3 \sqrt{3}}{2^{14/3} \pi ^2}\left[\dfrac{\Gamma \left(\frac{1}{3}\right)}{\sqrt{\pi }}\right]^6+O(4-u),\\[12pt]\dfrac{1}{4-u}g_3\left(-\dfrac{108 u}{(4-u)^3}\right)& \dfrac{\pi ^{5/2}}{2^{5/3} \sqrt{3}}\left[\dfrac{\Gamma \left(\frac{1}{3}\right)}{\sqrt{\pi }}\right]^6-\dfrac{\pi ^{5/2} \sqrt{4-u}}{\sqrt{3}}+O(4-u).\end{array}\right.
\label{eq:near4}\end{align}As we may recall, soon after the following evaluation \begin{align}
\int_{0}^\infty K_{0}(2t)[I_{0}(t)]^{2} [K_0(t)]^2t\D t=\frac{\pi }{2^{20/3}}\left[\dfrac{\Gamma \left(\frac{1}{3}\right)}{\sqrt{\pi }}\right]^6\label{eq:Broadhurst2008id}
\end{align} had been conjectured by Bailey--Borwein--Broadhurst--Glasser \cite[][(101)]{BBBG2008}, the same was verified by Broadhurst \cite{Broadhurst2008}.  Now that $ \int_{0}^\infty K_{0}(\sqrt{u}t)[I_{0}(t)]^{2} [K_0(t)]^2t\D t$ admits a Taylor expansion in a neighborhood of $u=4$, with its leading coefficient given by the right-hand side of \eqref{eq:Broadhurst2008id}, we must have $ A=-\frac{\pi^2}{4},B=0$  in \eqref{eq:constAB}, thereby proving \eqref{eq:MB2'}.     \end{proof}\begin{remark}For completeness, we give another proof of   \eqref{eq:Broadhurst2008id}, along with some generalizations. Our methods are largely  independent of those employed in \cite{Broadhurst2008}.

First, we note that the evaluation \begin{align}
\int_{0}^\infty I_{0}(2t)I_{0}(t)[K_0(t)]^3t\D t=\frac{3\pi }{2^{20/3}}\left[\dfrac{\Gamma \left(\frac{1}{3}\right)}{\sqrt{\pi }}\right]^6
\end{align} follows from \eqref{eq:IKM3F2} and the first line in \eqref{eq:near4}.

Then, for  $\ell\in\mathbb Z_{>0}$ and $\lambda,\mu\in(0,\infty)$, we consider a vanishing contour integral\begin{align}
\int_{-i\infty}^{i\infty}z[H_0^{(1)}(z)H_0^{(2)}(z)]^\ell H_0^{(1)}(\lambda z)H_0^{(1)}(\mu z)H_0^{(2)}((\lambda+\mu) z)\D z=0,
\end{align}where the contour closes to the right,  thanks to asymptotic expansions of the Hankel functions in the $ |z|\to\infty$ regime \cite[][\S7.2]{Watson1944Bessel}. Spelling out the Hankel functions along the imaginary axis in terms of modified Bessel functions, we arrive at a sum rule\begin{align}
&
i\int_{0}^\infty[K_0(t)]^\ell[\pi I_{0}(t)+iK_{0}(t)]^{\ell}[\pi I_{0}(\lambda t)+iK_{0}(\lambda t)][\pi I_{0}(\mu t)+iK_{0}(\mu t)]K_0((\lambda+\mu)t)t\D t\notag\\{}&-(-1)^{\ell}\int_{0}^\infty[K_0(t)]^{\ell}[\pi I_{0}(t)-iK_{0}(t)]^{\ell}K_0(\lambda t)K_0(\mu t)[\pi  I_{0}((\lambda+\mu)t)-iK_0((\lambda+\mu)t)]t\D t=0.\label{eq:ab_a+b_cancel}
\end{align} Setting $ \lambda=\mu=1$ in \eqref{eq:ab_a+b_cancel},  we obtain  a  cancelation formula that is valid for every $ \ell\in\mathbb Z_{>0}$:\begin{align}&
i\int_{0}^\infty[K_0(t)]^\ell[\pi I_{0}(t)+iK_{0}(t)]^{\ell+2}K_0(2t)t\D t\notag\\{}&-(-1)^{\ell}\int_{0}^\infty[K_0(t)]^{\ell+2}[\pi I_{0}(t)-iK_{0}(t)]^{\ell}[\pi  I_{0}(2t)-iK_0(2t)]t\D t=0.
\end{align}This incorporates \begin{align}
\int_{0}^\infty I_{0}(2t)I_{0}(t)[K_0(t)]^3t\D t=3\int_{0}^\infty K_{0}(2t)[I_{0}(t)]^{2} [K_0(t)]^2t\D t
\end{align} as a special case (real part for $\ell=1$).  \eor\end{remark}
  \subsection{Representations for $ \IKM(1,5;3)$}Towards our goal of proving \begin{align}
\int_0^\infty I_0(t)[K_0(t)]^5t(1-8t^{2})\D t={}&\frac{7 \pi ^3}{108 \sqrt{3}}\int_{0}^1\left[ _2F_1\left(\left.\begin{array}{c}
-\frac{1}{3},\frac{1}{3} \\[4pt]
1 \\
\end{array}  \right|x\right)\right]^2\frac{\D x}{\sqrt{1-x}},
\end{align}we begin with two lemmata concerning diagrams of sunrise type, namely,\begin{align}
\int_0^\infty I_0(t)[K_0(t)]^5t^{2m+1}\D t
\end{align}for $ m\in\{0,1,2\}$.\begin{lemma}[Alternative integral representations for $ \IKM(1,5;3)$ and $ \IKM(1,5;5)$]We have the following identities:\begin{align}
\int_0^\infty I_0(t)[K_0(t)]^5t^{3}\D t={}&\frac{\pi^2}{3}\int_0^\infty I_0(t)K_0(t)\left\{[I_{0}(t)]^{2} [K_0(t)]^2-\frac{1}{4t^{2}}\right\}t^{3}\D t,\label{eq:IKM333IKM153}\\\int_0^\infty I_0(t)[K_0(t)]^5t^{5}\D t={}&\frac{\pi^2}{3}\int_0^\infty I_0(t)K_0(t)\left\{[I_{0}(t)]^{2} [K_0(t)]^2-\frac{1}{4t^{2}}-\frac{1}{16 t^4}\right\}t^{5}\D t.\label{eq:IKM335IKM155}
\end{align}\end{lemma}\begin{proof}As a variation upon \cite[][(3.13)]{HB1}, we study a vanishing contour integral\begin{align}
\lim_{T\to\infty}\int_{-iT}^{iT}H_0^{(1)}(z)H_0^{(2)}(z)\left\{[H_0^{(1)}(z)H_0^{(2)}(z)]^{2}-\frac{4}{\pi^{2}z^2}\right\}z^3\D z=0,
\end{align}where $ H_0^{(1)}(z)$ and $H_0^{(2)}(z)$ are cylindrical Hankel functions. By pairing up the integrand at $ z=it$ and $z=-it$, and using the fact that\begin{align}
H_0^{(1)}(it)H_0^{(2)}(it)=\frac{4K_{0}(|t|)}{\pi^{2}}\left[ K_0(|t|)-\frac{\pi it}{|t|}I_0(|t|) \right],\quad \forall t\in(-\infty,0)\cup(0,\infty),
\end{align}we may reduce the vanishing contour integral into our claimed result in \eqref{eq:IKM333IKM153}.

The proof of \eqref{eq:IKM335IKM155} founds on a similar principle.\end{proof}
\begin{lemma}[A sum rule for Bessel moments]We have the following vanishing identity:\begin{align}
\int_0^\infty I_0(t)[K_0(t)]^5t(2-85t^{2}+72t^{4})\D t=0.\label{eq:IKM151IKM153IKM155}
\end{align}\end{lemma}\begin{proof}In \cite[][\S6.2]{BBBG2008}, Bailey--Borwein--Broadhurst--Glasser reported that \eqref{eq:IKM151IKM153IKM155}  is correct up to 1200 decimal places. We now prove this sum rule using Vanhove's fourth-order differential operator \cite[][Table 1, $n=5$]{Vanhove2014Survey}\begin{align}\begin{split} \widetilde L_4:={}&
u^2(u-25) (u-9) (u-1) \frac{\D^4}{\D u^4}+2 u (5 u^3-140 u^2+777 u-450)\frac{\D^3}{\D u^3}\\{}&+(25 u^3-518 u^2+1839 u-450)\frac{\D^2}{\D u^2}\\{}&+(3 u-5) (5 u-57)\frac{\D}{\D u}+(u-5),\label{eq:VL4}
\end{split}\end{align}which satisfies $ \widetilde L_4\int_0^\infty I_0(\sqrt{u}t)K_0(t)[K_0(t)]^4t\D t=-\frac{15}{2}$  \cite[][Lemma 4.2]{Zhou2017BMdet}. Differentiating under the integral sign in the identity below, \begin{align}
\frac{\D}{\D u}\left\{ \widetilde L_4\int_0^\infty I_0(\sqrt{u}t)K_0(t)[K_0(t)]^4t\D t\right\}=0,
\end{align}before specializing to $u=1$, we arrive at $ \frac{1}{2}\int_0^\infty I_0(t)[K_0(t)]^5t(2-85t^{2}+72t^{4})\D t=0$, as claimed.\end{proof}\begin{remark}In \cite[][\S6.1]{BBBG2008}, Bailey--Borwein--Broadhurst--Glasser reported that $ \int_0^\infty [I_0(t)]^{2}[K_0(t)]^4t(2-85t^{2}+72t^{4})\D t=0$  is correct up to 1200 decimal places. This sum rule can be proved by a similar procedure as in the lemma above, namely, by considering\begin{align}
\frac{\D}{\D u} \left\{\widetilde L_4\int_0^\infty I_0(\sqrt{u}t)I_0(t)[K_0(t)]^4t\D t\right\}=0
\end{align}at $u=1$.  \eor\end{remark}

\begin{proposition}[Mellin--Barnes integrals for $ \IKM(1,5;2n+1),n\in\{0,1,2\}$]\label{prop:IKM15n_MB}\begin{enumerate}[leftmargin=*,  label=\emph{(\alph*)},ref=(\alph*),
widest=a, align=left]\item Setting\begin{align}
\varPhi(s):=\frac{\pi^{3}\Gamma \left(\frac{1}{3}-s\right)  \Gamma \left(\frac{2}{3}-s\right) \Gamma \left(s-\frac{1}{6}\right) \Gamma \left(s+\frac{1}{6}\right)}{72 \sqrt{3} [  \Gamma (1-s)]^2\left[\Gamma \left(s+\frac{1}{2}\right)\right]^2},
\end{align}we have{\allowdisplaybreaks\begin{align}\begin{split}&\int_0^\infty I_0(t)[K_0(t)]^5t\D t\equiv \IKM(1,5;1)\\={}&\frac{1}{2\pi i}\int_{\frac14-i\infty}^{\frac14+i\infty}\varPhi(s)\D s,\end{split}\label{eq:IKM151_MB}\\\begin{split}&\int_0^\infty I_0(t)[K_0(t)]^5t^{3}\D t\equiv \IKM(1,5;3)\\={}&\frac{1}{2\pi i}\int_{\frac14-i\infty}^{\frac14+i\infty}\varPhi(s)\left[\frac{1}{3 (5-6 s)}+\frac{1}{2 s+1}-\frac{2}{3}\right]\D s-\frac{2 \pi ^{5/2}}{27 \sqrt{3}}\left[\frac{\sqrt{\pi}}{\Gamma\left( \frac{1}{3} \right)}\right]^{9},\end{split}\label{eq:IKM153_MB}\\\begin{split}&\int_0^\infty I_0(t)[K_0(t)]^5t^{5}\D t\equiv \IKM(1,5;5)\\={}&\frac{1}{2\pi i}\int_{\frac14-i\infty}^{\frac14+i\infty}\varPhi(s)\left[\frac{25}{54 (7-6 s)}+\frac{43}{108 (5-6 s)}+\frac{23}{4 (2 s+1)}-\frac{45}{2 (2 s+3)}+\frac{68}{27} \right]\D s\\{}&-\frac{43 \pi ^{5/2} }{486 \sqrt{3}}\left[\frac{\sqrt{\pi }}{\Gamma \left(\frac{1}{3}\right)}\right]^9-\frac{5 \pi ^{5/2} }{331776}\left[\frac{\Gamma \left(\frac{1}{3}\right)}{\sqrt{\pi }}\right]^9.\end{split}\label{eq:IKM155_MB}
\end{align}}Moreover, we have  the following vanishing identity:\begin{align}\begin{split}0={}&
\frac{1}{2\pi i}\int_{\frac14-i\infty}^{\frac14+i\infty}\varPhi(s)\left[\frac{100}{3 (7-6 s)}+\frac{1}{3 (5-6 s)}+\frac{329}{2 s+1}-\frac{1620}{2 s+3}+240 \right]\D s\\&-\frac{5 \pi ^{5/2} }{4608}\left[\frac{\Gamma \left(\frac{1}{3}\right)}{\sqrt{\pi }}\right]^9-\frac{2 \pi ^{5/2}}{27 \sqrt{3} }\left[\frac{\sqrt{\pi }}{\Gamma \left(\frac{1}{3}\right)}\right]^9.\end{split}\label{eq:Phi_vanish}
\end{align}\item We have\begin{align}
\frac{1}{2\pi i}\int_{\frac14-i\infty}^{\frac14+i\infty}\varPhi(s)\left(\frac{1}{5-6s}-\frac{2}{1+2s}+1\right)\D s=\frac{2 \pi ^{5/2}}{9 \sqrt{3}}\left[\frac{\sqrt{\pi}}{\Gamma\left( \frac{1}{3} \right)}\right]^{9},
\end{align}which entails\begin{align}\int_0^
\infty I_0(t)[K_0(t)]^5t^{3}\D t=\frac{1}{2\pi i}\int_{\frac14-i\infty}^{\frac14+i\infty}\varPhi(s)\left[\frac{5}{6(5-6s)}-\frac{1}{6}\right]\D s-\frac{5 \pi ^{5/2}}{27 \sqrt{3}}\left[\frac{\sqrt{\pi}}{\Gamma\left( \frac{1}{3} \right)}\right]^{9}.
\end{align} \end{enumerate} \end{proposition}
\begin{proof}\begin{enumerate}[leftmargin=*,  label=(\alph*),ref=(\alph*),
widest=a, align=left]\item One can verify \eqref{eq:IKM151_MB} by counting the residues at $ s=n+\frac{1}{3},n+\frac{2}{3}$ for $ n\in\mathbb Z_{\geq0}$, before comparing to the hypergeometric identity in \eqref{eq:IKM151_4F3_sum}. Here, contour closure is permissible, due to the leading asymptotic behavior\begin{align}
\varPhi(s)\sim\frac{\pi ^3 }{72 \sqrt{3} s^2}\left[\frac{3}{2 \cos (2 \pi  s)+1}-1\right],\quad s\to\infty.
\end{align}

We begin our treatment of \eqref{eq:IKM153_MB} with
 an analog of the Neumann addition formula in \eqref{eq:I0_add}, namely\begin{align}
I_{0}(t)K_0(t)=\frac{2}{\pi}\int_0^\pi K_0(2t\cos\theta)\D\theta,\label{eq:K0_add}
\end{align} as well as an integral formula $ \int_0^\infty K_{0}(\sqrt{u}t)t\D t=\frac{1}{u}$ for $u>0$ \cite[][\S13.21(8)]{Watson1944Bessel}, which lead us to\begin{align}\begin{split}&
\pi^2\int_0^\infty I_0(t)K_0(t)\left\{[I_{0}(t)]^{2} [K_0(t)]^2-\frac{1}{4t^{2}}\right\}t^{3}\D t\\={}&\pi\int_{0}^\infty\left(\int_0^4\left\{[I_{0}(t)]^{2} [K_0(t)]^2-\frac{1}{4t^{2}}\right\}\frac{K_{0}(\sqrt{u}t)\D u}{\sqrt{u(4-u)}}\right)t^{3}\D t\\={}&\pi\int_0^4\left\{\int_{0}^\infty K_{0}(\sqrt{u}t)[I_{0}(t)]^{2} [K_0(t)]^2t^{3}\D t-\frac{1}{4u}\right\}\frac{\D u}{\sqrt{u(4-u)}}.\end{split}\label{eq:IKM153prep}
\end{align}

Writing\begin{align}\begin{split}
\varphi(u,s):={}&\frac{\sqrt{3}}{8\pi ^{3/2}}\frac{\Gamma \left(\frac{1}{3}-s\right) \Gamma \left(\frac{1}{2}-s\right) \Gamma \left(\frac{2}{3}-s\right) \Gamma (s)}{4-u}\left[\frac{108u}{(4-u)^3}\right]^{-s}\left\{ [\Gamma(s)]^2 -\frac{\pi^2}{[\Gamma (1-s)]^2} \right\}\\={}&\frac{\sqrt{3 \pi }}{8}\frac{\cot ^2(\pi  s) \Gamma \left(\frac{1}{3}-s\right) \Gamma \left(\frac{1}{2}-s\right) \Gamma \left(\frac{2}{3}-s\right) \Gamma (s)}{(4-u)[ \Gamma (1-s)]^2}\left[\frac{108u}{(4-u)^3}\right]^{-s}\end{split}
\end{align}for $ u\in(0,4)$, while referring to  \eqref{eq:IvKM231_MB} and \eqref{eq:MB2}, we obtain\begin{align}\begin{split}&
\int_{0}^\infty K_{0}(\sqrt{u}t)[I_{0}(t)]^{2} [K_0(t)]^2t\D t=\frac{1}{2\pi i}\int_{\frac{1}{4}-i\infty}^{\frac{1}{4}+i\infty}\varphi(u,s)\D s\\={}&\frac{1}{2\pi i}\int_{\frac{1}{4}-i\infty}^{\frac{1}{4}+i\infty}\varphi(u,s)\left[1-\left( \frac{4-u}{4} \right)^{1-3s}\left(1-\frac{3 s-1}{4} u\right)\right]\D s\\{}&+\frac{1}{2\pi i}\int_{\frac{1}{4}-i\infty}^{\frac{1}{4}+i\infty}\varphi(u,s)\left( \frac{4-u}{4} \right)^{1-3s}\left(1-\frac{3 s-1}{4} u\right)\D s.
\end{split}\end{align}Shifting  contours while picking up residues, we arrive at a decomposition\begin{align}\begin{split}
&\int_{0}^\infty K_{0}(\sqrt{u}t)[I_{0}(t)]^{2} [K_0(t)]^2t\D t
\\={}&\frac{1}{2\pi i}\int_{\frac{5}{4}-i\infty}^{\frac{5}{4}+i\infty}\varphi(u,s)\left[1-\left( \frac{4-u}{4} \right)^{1-3s}\left(1-\frac{3 s-1}{4} u\right)\right]\D s\\{}&+\frac{1}{2\pi i}\int_{-\frac{3}{4}-i\infty}^{-\frac{3}{4}+i\infty}\varphi(u,s)\left( \frac{4-u}{4} \right)^{1-3s}\D s-\frac{1}{2\pi i}\int_{\frac{1}{4}-i\infty}^{\frac{1}{4}+i\infty}\varphi(u,s)\left( \frac{4-u}{4} \right)^{1-3s}\frac{3 s-1}{4} u\D s\\{}&+\frac{\log u}{32}\log \frac{u}{4096}+\frac{\pi^{2}}{96}+\frac{9 \log ^22}{8}.\end{split}\end{align}

Consequently,  in view of the Bessel differential equation  $ \left(u\frac{\partial^2}{\partial u^2}+\frac{\partial}{\partial u}\right)K_0(\sqrt{u}t)=\frac{t^2}{4}K_0(\sqrt{u}t)$, we have the following identity for $u\in(0,4) $:\begin{align}\begin{split}&
\int_{0}^\infty K_{0}(\sqrt{u}t)[I_{0}(t)]^{2} [K_0(t)]^2t^{3}\D t-\frac{1}{4u}\\={}&\frac{4}{2\pi i}\int_{\frac{5}{4}-i\infty}^{\frac{5}{4}+i\infty}\left(u\frac{\partial^2}{\partial u^2}+\frac{\partial}{\partial u}\right)\left\{\varphi(u,s)\left[1-\left( \frac{4-u}{4} \right)^{1-3s}\left(1-\frac{3 s-1}{4} u\right)\right]\right\}\D s\\{}&+\frac{4}{2\pi i}\int_{-\frac{3}{4}-i\infty}^{-\frac{3}{4}+i\infty}\left(u\frac{\partial^2}{\partial u^2}+\frac{\partial}{\partial u}\right)\left[\varphi(u,s)\left( \frac{4-u}{4} \right)^{1-3s}\right]\D s\\{}&-\frac{4}{2\pi i}\int_{\frac{1}{4}-i\infty}^{\frac{1}{4}+i\infty}\left(u\frac{\partial^2}{\partial u^2}+\frac{\partial}{\partial u}\right)\left[\varphi(u,s)\left( \frac{4-u}{4} \right)^{1-3s}\frac{3 s-1}{4} u\right]\D s.\end{split}
\end{align}  Now, we complete the integration over $u$ in \eqref{eq:IKM153prep}, by applying the explicit formula for $\varphi(u,s)$ to the equation above, and invoking the Fubini theorem for exchanging the order of integrations in absolutely convergent double integrals.  The result reads{\allowdisplaybreaks\begin{align}\begin{split}{}&
\int_0^4\left\{\int_{0}^\infty K_{0}(\sqrt{u}t)[I_{0}(t)]^{2} [K_0(t)]^2t^{3}\D t-\frac{1}{4u}\right\}\frac{\D u}{\sqrt{u(4-u)}}\\={}&-\frac{1}{2\pi i}\int_{\frac{5}{4}-i\infty}^{\frac{5}{4}+i\infty}\frac{\sqrt{3} \pi  \Gamma \left(\frac{1}{3}-s\right) \Gamma \left(\frac{2}{3}-s\right) [\Gamma (s)] ^3}{2 \Gamma \left(s-\frac{1}{2}\right) \Gamma \left(s+\frac{3}{2}\right)}\left\{- \frac{2^{2 (s-2)}}{3^{3 s}}\frac{3 s^3-5 s^2-s+1}{\Gamma (1-s)}+\right.\\{}&\left.+\frac{(2s-1)\pi \cos ( \pi  s)}{36 \cos (3 \pi  s)}\frac{1}{\Gamma \left(\frac{7}{6}-s\right) \Gamma \left(\frac{11}{6}-s\right) \Gamma (s)}\right\}\D s\\{}&+\frac{1}{2\pi i}\int_{-\frac{3}{4}-i\infty}^{-\frac{3}{4}+i\infty}\frac{ 2^{2 (s-2)}}{3^{3 s}}\frac{\sqrt{3}\pi  \Gamma \left(\frac{1}{3}-s\right) \Gamma \left(\frac{2}{3}-s\right) [\Gamma (s)]^3}{2 \Gamma (1-s) \Gamma \left(s+\frac{1}{2}\right) \Gamma \left(s+\frac{3}{2}\right)}s^{3}\D s\\{}&-\frac{1}{2\pi i}\int_{\frac{1}{4}-i\infty}^{\frac{1}{4}+i\infty}\frac{ 2^{2 (s-2)}}{3^{3 s}}\frac{\sqrt{3}\pi  \Gamma \left(\frac{1}{3}-s\right) \Gamma \left(\frac{2}{3}-s\right) [\Gamma (s)]^3}{4 \Gamma (1-s) \Gamma \left(s+\frac{1}{2}\right) \Gamma \left(s+\frac{3}{2}\right)}(s-1)^2 (6 s^2+s-1)\D s.\end{split}
\end{align}}We can shift the contour of the penultimate integral to $\R s=\frac14 $, without encountering any singularities on the way. This further allows us to combine the last two integrals, and turn the expression above into  \begin{align}\begin{split}
&-\frac{1}{2\pi i}\int_{\frac{5}{4}-i\infty}^{\frac{5}{4}+i\infty}\frac{\sqrt{3} \pi  \Gamma \left(\frac{1}{3}-s\right) \Gamma \left(\frac{2}{3}-s\right) [\Gamma (s)] ^3}{2 \Gamma \left(s-\frac{1}{2}\right) \Gamma \left(s+\frac{3}{2}\right)}\left\{- \frac{2^{2 (s-2)}}{3^{3 s}}\frac{3 s^3-5 s^2-s+1}{\Gamma (1-s)}+\right.\\{}&\left.+\frac{(2s-1)\pi \cos ( \pi  s)}{36 \cos (3 \pi  s)}\frac{1}{\Gamma \left(\frac{7}{6}-s\right) \Gamma \left(\frac{11}{6}-s\right) \Gamma (s)}\right\}\D s\\{}&-\frac{1}{2\pi i}\int_{\frac{1}{4}-i\infty}^{\frac{1}{4}+i\infty}\frac{\sqrt{3} \pi  \Gamma \left(\frac{1}{3}-s\right) \Gamma \left(\frac{2}{3}-s\right) [\Gamma (s)] ^3}{2 \Gamma \left(s-\frac{1}{2}\right) \Gamma \left(s+\frac{3}{2}\right)}\frac{2^{2 (s-2)}}{3^{3 s}}\frac{3 s^3-5 s^2-s+1}{\Gamma (1-s)}\D s.\end{split}\end{align}Counting residues at $s=\frac13 $ and $s=\frac23$ in the last integrand, we may further simplify our result into\begin{align}\begin{split}&\frac{3}{\pi}\int_0^\infty I_0(t)[K_0(t)]^5t^{3}\D t\\={}&
\pi\int_0^\infty I_0(t)K_0(t)\left\{[I_{0}(t)]^{2} [K_0(t)]^2-\frac{1}{4t^{2}}\right\}t^{3}\D t\\={}&\frac{1}{2\pi i}\int_{\frac{5}{4}-i\infty}^{\frac{5}{4}+i\infty}\frac{\pi ^2\cos (\pi  s) }{24 \sqrt{3}  \cos (3 \pi  s)}\frac{(1-2 s) \Gamma \left(\frac{1}{3}-s\right) \Gamma \left(\frac{2}{3}-s\right) [\Gamma (s)]^2}{ \Gamma \left(\frac{7}{6}-s\right) \Gamma \left(\frac{11}{6}-s\right) \Gamma \left(s-\frac{1}{2}\right) \Gamma \left(s+\frac{3}{2}\right)}\D s\\{}&+\frac{\pi ^{3/2} }{1920}\left[\frac{\Gamma \left(\frac{1}{3}\right)}{\sqrt{\pi }}\right]^9-\frac{8 \pi ^{3/2} }{21 \sqrt{3}}\left[\frac{\sqrt{\pi }}{\Gamma \left(\frac{1}{3}\right)}\right]^9.\end{split}
\end{align} Summing over residues at $ s=n+\frac{1}{3}$ and $ s=n+\frac{2}{3}$ for $ n\in\mathbb Z_{>0}$, we can evaluate  last formula as\begin{align}\begin{split}&
\frac{\pi ^{3/2}}{1920}\left[\frac{\Gamma \left(\frac{1}{3}\right)}{\sqrt{\pi }}\right]^9{_4F_3}\left(\left. \begin{array}{c}-\frac{1}{2},\frac{1}{6},\frac{1}{3},\frac{1}{3}\\[4pt]-\frac{1}{6},\frac{2}{3},\frac{11}{6}\end{array} \right|1\right)-\frac{8 \pi ^{3/2} }{21 \sqrt{3}}\left[\frac{\sqrt{\pi }}{\Gamma \left(\frac{1}{3}\right)}\right]^9{_4F_3}\left(\left. \begin{array}{c}-\frac{1}{6},\frac{1}{2},\frac{2}{3},\frac{2}{3}\\[4pt]\frac{1}{6},\frac{4}{3},\frac{13}{6}\end{array} \right|1\right)\\{}&-\frac{\pi ^{3/2}}{7040}\left[\frac{\Gamma \left(\frac{1}{3}\right)}{\sqrt{\pi }}\right]^9{_4F_3}\left(\left. \begin{array}{c}\frac{1}{2},\frac{7}{6},\frac{4}{3},\frac{4}{3}\\[4pt]\frac{5}{6},\frac{5}{3},\frac{17}{6}\end{array} \right|1\right)+\frac{16 \pi ^{3/2}}{91 \sqrt{3}}\left[\frac{\sqrt{\pi }}{\Gamma \left(\frac{1}{3}\right)}\right]^9{_4F_3}\left(\left. \begin{array}{c}\frac{5}{6},\frac{3}{2},\frac{5}{3},\frac{5}{3}\\[4pt]\frac{7}{6},\frac{7}{3},\frac{19}{6}\end{array} \right|1\right).
\end{split}\end{align}The same sum of hypergeometric series is also produced by the following expression:\begin{align}\frac{3}{\pi}
\left\{\frac{1}{2\pi i}\int_{\frac14-i\infty}^{\frac14+i\infty}\varPhi(s)\left[\frac{1}{3 (5-6 s)}+\frac{1}{2 s+1}-\frac{2}{3}\right]\D s-\frac{2 \pi ^{5/2}}{27 \sqrt{3}}\left[\frac{\sqrt{\pi}}{\Gamma\left( \frac{1}{3} \right)}\right]^{9}\right\},
\end{align} because the trailing constant cancels out the residue at $ s=\frac{5}{6}$, and  the series expansions agree, term by term, with  the residue contributions  at  the poles $ s=n+\frac{1}{3},n+\frac{2}{3}$ for $ n\in\mathbb Z_{\geq0}$.
Thus, we have confirmed \eqref{eq:IKM153_MB}.

With essentially the same set of ideas, we can use \eqref{eq:IKM335IKM155} to demonstrate  \eqref{eq:IKM155_MB}.

Transcribing \eqref{eq:IKM151IKM153IKM155} using    \eqref{eq:IKM151_MB}--\eqref{eq:IKM155_MB}, we arrive at \eqref{eq:Phi_vanish}.\item\label{itm:Cstar} To facilitate further analysis, we write $C_*$ for the union of infinitesimal clockwise circular contours centered at $ \left\{\left.n+\frac{1}{3}\right|n\in\mathbb Z_{\geq0}\right\}\cup\left\{\left.n+\frac{2}{3}\right|n\in\mathbb Z_{\geq0}\right\}$. This notation allows us to compress the right-hand sides of \eqref{eq:IKM153_MB}--\eqref{eq:Phi_vanish} into the form\begin{align}
\frac{1}{2\pi i}\int_{C_*}\varPhi(s)[\cdots]\D s,
\end{align}without the trailing constants.

Equipped with the reflection formula  $ \varPhi(s)=\varPhi\left( \frac{1}{2}-s \right)$ and the recursion for Euler's gamma function, we have \begin{align}\begin{split}&
\frac{1}{2\pi i}\int_{C_*}\varPhi(s)\left(\frac{1}{7-6 s}-\frac{1}{4}\right)\D s\\={}&\frac{1}{2\pi i}\int_{C_*}\varPhi(s)\left[\frac{6}{7-6 s}+\frac{1}{2 (5-6 s)}-\frac{3}{4 (1-s)}-\frac{1}{4}\right]\D s\end{split}
\end{align} upon a reflection $ s\mapsto \frac32-s$,  which subsequently rearranges to\begin{align}\begin{split}0={}&
\frac{1}{2\pi i}\int_{C_*}\varPhi(s)\left[\frac{5}{7-6 s}+\frac{1}{2 (5-6 s)}-\frac{3}{4 (1-s)}\right]\D s\\={}&\frac{1}{2\pi i}\int_{C_*}\varPhi(s)\left[\frac{5}{7-6 s}+\frac{1}{2 (5-6 s)}-\frac{3}{2 (1+2s)}\right]\D s.\label{eq:Phi_cancel1}\end{split}
\end{align}Here, in the last step, we have applied the reflection $ s\mapsto\frac{1}{2}-s$ to the last summand of the integrand. Likewise, by reflection and rearrangements, we obtain\begin{align}\begin{split}&
\frac{1}{2\pi i}\int_{C_*}\varPhi(s)\left(\frac{1}{2}-\frac{1}{1+2s}\right)\D s\\={}&\frac{1}{2\pi i}\int_{C_*}\varPhi(s)\left[ \frac{12}{7-6 s}+\frac{3}{7 (5-6 s)}-\frac{1}{1-s}-\frac{18}{7 (2-s)}+\frac{1}{2} \right]\D s\end{split}
\end{align}and its equivalent form\begin{align}
0=\frac{1}{2\pi i}\int_{C_*}\varPhi(s)\left[ \frac{12}{7-6 s}+\frac{3}{7 (5-6 s)}-\frac{1}{1+2s}-\frac{36}{7 (3+2s)} \right]\D s.\label{eq:Phi_cancel2}
\end{align}Using \eqref{eq:Phi_cancel1} and \eqref{eq:Phi_cancel2}, we can eliminate the terms related to $ \frac{1}{7-6s}$ and $ \frac{1}{3+2s}$ from  \eqref{eq:Phi_vanish}, which brings us \begin{align}
0=\frac{1}{2\pi i}\int_{C_*}\varPhi(s)\left(\frac{1}{5-6s}-\frac{2}{1+2s}+1\right)\D s.
\end{align}
Employing the equation above, we  rewrite  \eqref{eq:IKM153_MB} as\begin{align}
\int_0^\infty I_0(t)[K_0(t)]^5t^{3}\D t=\frac{1}{2\pi i}\int_{C_*}\varPhi(s)\left[\frac{5}{6(5-6s)}-\frac{1}{6}\right]\D s.\label{eq:IKM153_Cstar}
\end{align}

Replacing the contour $C_*$ by the vertical line running from $\frac14-i\infty $ to $\ \frac14+i\infty$, we can convert the last two displayed equations into the claimed identities. \qedhere\end{enumerate}\end{proof}\begin{remark}At an earlier stage of the current work, we attempted to retrieve \eqref{eq:IKM153_MB} from the ``finite part'' of the following divergent integral:\begin{align}
\text{``}\frac{4\pi}{3}\int_0^4\left[ \left(u\frac{\partial^2}{\partial u^2}+\frac{\partial}{\partial u}\right)\int_0^\infty I_0(\sqrt{u}t)I_0(t)[K_0(t)]^3t\D t \right]\frac{\D u}{\sqrt{u(4-u)}}\text{''}.
\end{align}Our previous ``renormalized'' calculations began with an expression for the indefinite integral\begin{align}\begin{split}&
\frac{4\pi}{3}\int \left(u\frac{\partial^2}{\partial u^2}+\frac{\partial}{\partial u}\right)\left\{\frac{1}{4-u}\left[\frac{u}{(4-u)^3}\right]^{-s}\right\}\frac{\D u}{\sqrt{u(4-u)}}\\={}&\frac{  4^{2 s-1}\pi}{3}  \left[s^2 \mathrm B_{\frac{u}{4}}\left(-s-\frac{1}{2},3 s-\frac{5}{2}\right)+(s-1) (4 s-1)\mathrm  B_{\frac{u}{4}}\left(\frac{1}{2}-s,3 s-\frac{5}{2}\right)\right.\\{}&\left.+(1-2 s)^2 \mathrm B_{\frac{u}{4}}\left(\frac{3}{2}-s,3 s-\frac{5}{2}\right)\right]\end{split}
\end{align}in terms of incomplete beta functions, which are analytic continuations of $\mathrm B_z(a,b):=\int_0^zt^{a-1}(1-t)^{b-1}\D t $ for $ \R a>1$. We then forcibly set $u=4$ in the indefinite integral, and referred back to the Mellin--Barnes representation for $ \int_0^\infty I_0(\sqrt{u}t)I_0(t)[K_0(t)]^3t\D t$ in \eqref{eq:IvKM231_MB}, before arriving at the integrand in  \eqref{eq:IKM153_MB}. Later afterwards, we found that such formal arguments can be turned to rigorous computations, with appropriate subtractions and contour shifts before invocations of the Fubini theorem, as described in the proof above.   \eor\end{remark}

  With the foregoing preparations, we can prove the integral identity announced in \eqref{eq:IKM151_IKM153_diff}.
\begin{proposition}[Meijer reduction of the Broadhurst--Laporta integral]\label{prop:IKM15n_Meijer}\begin{enumerate}[leftmargin=*,  label=\emph{(\alph*)},ref=(\alph*),
widest=a, align=left]\item We have \begin{align}\begin{split}{}&
\int_{0}^1\left[ _2F_1\left(\left.\begin{array}{c}
-\frac{1}{3},\frac{1}{3} \\[4pt]
1 \\
\end{array}  \right|x\right)\right]^2\frac{\D x}{\sqrt{1-x}}\\={}&G_{4,4}^{2,2}\left(1\left|
\begin{array}{c}
 -\frac{1}{2},\frac{1}{2},\frac{2}{3},\frac{4}{3} \\[4pt]
 0,1,-\frac{5}{6},-\frac{1}{6} \\
\end{array}
\right.\right)=-\frac{3}{4 \pi ^2}G_{4,4}^{2,4}\left(1\left|
\begin{array}{c}
 -\frac{1}{2},\frac{1}{2},\frac{2}{3},\frac{4}{3} \\[4pt]
 0,1,-\frac{5}{6},-\frac{1}{6} \\
\end{array}
\right.\right)\\={}&\frac{3}{\sqrt{\pi }}\left\{\frac{\sqrt{3} }{2^7}\left[\frac{\Gamma \left(\frac{1}{3}\right)}{\sqrt{\pi }}\right]^9\, _4F_3\left(\left. \begin{array}{c}-\frac{1}{2},\frac{1}{6},\frac{1}{3},\frac{4}{3}\\[4pt]-\frac{1}{6},\frac{5}{6},\frac{5}{3}\end{array} \right|1\right)+\frac{5}{7}\frac{2^4}{3}  \left[\frac{\sqrt{\pi }}{\Gamma \left(\frac{1}{3}\right)}\right]^9\, _4F_3\left(\left. \begin{array}{c}-\frac{7}{6},-\frac{1}{2},-\frac{1}{3},\frac{2}{3}\\[4pt]-\frac{5}{6},\frac{1}{6},\frac{1}{3}\end{array} \right|1\right) \right\}\\={}&\frac{6561}{3850}{_7F_6}\left(\left. \begin{array}{c}
-\frac{1}{3},\frac{1}{3},\frac{2}{3},\frac{4}{3},\frac{3}{2},\frac{3}{2},\frac{7}{4} \\[4pt]
\frac{3}{4},1,\frac{7}{6},\frac{11}{6},\frac{13}{6},\frac{17}{6} \\
\end{array} \right|1\right).\end{split}\label{eq:BroadhurstLaportaMeijer}
\end{align}\item The following identity holds:\begin{align}
\int_0^\infty I_0(t)[K_0(t)]^5t(1-8t^{2})\D t={}&\frac{7 \pi ^3}{108 \sqrt{3}}\int_{0}^1\left[ _2F_1\left(\left.\begin{array}{c}
-\frac{1}{3},\frac{1}{3} \\[4pt]
1 \\
\end{array}  \right|x\right)\right]^2\frac{\D x}{\sqrt{1-x}}.
\end{align}\end{enumerate}\end{proposition}
\begin{proof}\begin{enumerate}[leftmargin=*,  label=(\alph*),ref=(\alph*),
widest=a, align=left]\item Using  a hypergeometric identity\begin{align}{_2F_1}\left(\left.\begin{array}{c}
-\nu,\nu\ \\[4pt]
1 \\
\end{array}  \right|1-t\right )=t\left( 1+\frac{1-t}{\nu}\frac{\D}{\D t} \right){_2F_1}\left(\left.\begin{array}{c}
-\nu,\nu+1\ \\[4pt]
1 \\
\end{array}  \right|1-t\right ),\end{align}we can deduce a Mellin transform formula\begin{align}
\int_0^1 {_2F_1}\left(\left.\begin{array}{c}
-\nu,\nu\ \\[4pt]
1 \\
\end{array}  \right|1-t\right )t^{s-1}\D t=\frac{\Gamma (s)\Gamma(s+1)}{\Gamma (s+1-\nu ) \Gamma (s+1+\nu )},\quad \R s>1.
\end{align}from \eqref{eq:Pnu_Mellin}. Consequently, Mellin convolution brings us  \begin{align}\begin{split}&
\int_0^1 \left[{_2F_1}\left(\left.\begin{array}{c}
-\nu,\nu\ \\[4pt]
1 \\
\end{array}  \right|1-t\right )\right]^2t^\alpha\D t\\={}&\frac{1}{2\pi i}\int_{\delta-i\infty}^{\delta+i\infty}\frac{\Gamma (\alpha +1-s)\Gamma (\alpha +2-s)\Gamma (s)\Gamma(s+1)\D s}{\Gamma (s+1-\nu ) \Gamma (s+\nu +1)\Gamma (\alpha +2- s-\nu ) \Gamma (\alpha +2-s+\nu)},\end{split}
\end{align}where $ \alpha\in(-1,\infty),\delta\in(0,\alpha+1)$. This incorporates the first equality in \eqref{eq:BroadhurstLaportaMeijer}, as a special case.

To prove the second equality in \eqref{eq:BroadhurstLaportaMeijer},  simply investigate the kernel space of the following differential operator:\begin{align}
z\left( z\frac{\D}{\D z}-\frac{1}{3} \right)\left( z\frac{\D}{\D z}+\frac{1}{3} \right)\left( z\frac{\D}{\D z}+\frac{1}{2} \right)\left( z\frac{\D}{\D z}+\frac{3}{2} \right)-\left(z\frac{\D}{\D z}-1\right)z\frac{\D}{\D z}\left( z\frac{\D}{\D z}+\frac{1}{6} \right)\left( z\frac{\D}{\D z}+\frac{5}{6} \right),
\end{align} in a similar fashion as its counterpart in Proposition \ref{prop:Laporta_Pnu_sqr}.

The third equality in \eqref{eq:BroadhurstLaportaMeijer} follows from residue calculus, as in the proof of the last equality in Proposition \ref{prop:Laporta_Pnu_sqr}.

To  prove the last equality in  \eqref{eq:BroadhurstLaportaMeijer},  simply set $ a= \frac{3}{2},b= \frac{3}{2},c= \frac{1}{3},d= -\frac{1}{3},e= \frac{4}{3},f= \frac{2}{3}$ in \eqref{eq:Bailey7F6}. (We note that there are actually 48 different choices of $ a,b,c,d,e,f$ in  Bailey's identity that fit the special $ G^{2,4}_{4,4}$ in question, producing four different $ _7F_6$ forms in total.  The other three expressions \begin{align}
\frac{81}{50}&{_7F_6}\left(\left. \begin{array}{c}
-\frac{1}{3},-\frac{1}{3},\frac{1}{3},\frac{1}{3},\frac{1}{2},\frac{1}{2},\frac{5}{4} \\[4pt]
\frac{1}{4},1,\frac{7}{6},\frac{7}{6},\frac{11}{6},\frac{11}{6} \\
\end{array} \right|1\right) ,\label{eq:vp7F6'}
\\\frac{567}{1375}\frac{\sqrt{3}}{2^{8/3}}\left[\frac{\Gamma \left(\frac{1}{3}\right)}{\sqrt{\pi }}\right]^6&{_7F_6}\left(\left. \begin{array}{c}
\frac{1}{3},\frac{1}{2},\frac{3}{2},\frac{3}{2},\frac{13}{6},\frac{13}{6},\frac{7}{3} \\[4pt]
\frac{7}{6},\frac{7}{6},\frac{11}{6},\frac{11}{6},\frac{17}{6},3 \\
\end{array} \right|1\right) ,\\\frac{27}{7}2^{5/3} \sqrt{3}\left[\frac{\sqrt{\pi }}{\Gamma \left(\frac{1}{3}\right)}\right]^6&{_5F_4}\left(\left. \begin{array}{c}
-\frac{1}{3},\frac{1}{2},\frac{3}{2},\frac{3}{2},\frac{5}{3} \\[4pt]
\frac{7}{6},\frac{7}{6},\frac{13}{6},3 \\
\end{array} \right|1\right)\label{eq:5F4_redn_a}\end{align}are all equal to the last line in  \eqref{eq:BroadhurstLaportaMeijer}, even though they had not been previously reported by Laporta or Broadhurst.)  \item Similar to what we did in the proof of Proposition \ref{prop:IKM15n_MB}\ref{itm:Cstar}, we now introduce another notation  $C_{**}$ for the union of infinitesimal clockwise circular contours centered at $ \left\{\left.n+\frac{1}{3}\right|n\in\mathbb Z_{\geq0}\right\}\cup\left\{\left.n-\frac{1}{3}\right|n\in\mathbb Z_{\geq0}\right\}$. By residue calculus, we can readily verify that \begin{align}\begin{split}&
\frac{3}{\sqrt{\pi }}\left\{\frac{\sqrt{3} }{2^7}\left[\frac{\Gamma \left(\frac{1}{3}\right)}{\sqrt{\pi }}\right]^9\, _4F_3\left(\left. \begin{array}{c}-\frac{1}{2},\frac{1}{6},\frac{1}{3},\frac{4}{3}\\[4pt]-\frac{1}{6},\frac{5}{6},\frac{5}{3}\end{array} \right|1\right)+\frac{5}{7}\frac{2^4}{3}  \left[\frac{\sqrt{\pi }}{\Gamma \left(\frac{1}{3}\right)}\right]^9\, _4F_3\left(\left. \begin{array}{c}-\frac{7}{6},-\frac{1}{2},-\frac{1}{3},\frac{2}{3}\\[4pt]-\frac{5}{6},\frac{1}{6},\frac{1}{3}\end{array} \right|1\right) \right\}\\={}&-\frac{3}{16 \pi ^2}\frac{1}{2\pi i}\int_{C_{**}}\frac{\Gamma \left(\frac{3}{2}-s\right) \Gamma \left(\frac{1}{2}-s\right) \Gamma \left(\frac{1}{3}-s\right) \Gamma \left(-\frac{1}{3}-s\right) \Gamma (s) \Gamma (s+1)}{\Gamma \left(\frac{7}{6}-s\right) \Gamma \left(\frac{11}{6}-s\right)}\D s.\end{split}\label{eq:Cstarstar_int}
\end{align}Since \begin{align}\begin{split}&
\frac{\Gamma \left(\frac{3}{2}-s\right) \Gamma \left(\frac{1}{2}-s\right) \Gamma \left(\frac{1}{3}-s\right) \Gamma \left(-s-\frac{1}{3}\right) \Gamma (s) \Gamma (s+1)}{\Gamma \left(\frac{7}{6}-s\right) \Gamma \left(\frac{11}{6}-s\right)}\\={}&\varPhi(s)\frac{72 \sqrt{3}  [1-2 \cos (2 \pi  s)] }{7\pi\sin ^2(2 \pi  s)}\left(\frac{10}{5-6 s}+\frac{5}{1+3 s}-7\right),
\end{split}\end{align}we can convert \eqref{eq:Cstarstar_int} into \begin{align}\begin{split}&
\int_{0}^1\left[ _2F_1\left(\left.\begin{array}{c}
-\frac{1}{3},\frac{1}{3} \\[4pt]
1 \\
\end{array}  \right|x\right)\right]^2\frac{\D x}{\sqrt{1-x}}\\={}&-\frac{36 \sqrt{3}}{7\pi ^3}\frac{1}{2\pi i}\int_{C_{**}}\varPhi(s)\left(\frac{10}{5-6 s}+\frac{5}{1+3 s}-7\right)\D s\\={}&-\frac{36 \sqrt{3}}{7\pi ^3}\frac{1}{2\pi i}\int_{C_{*}}\varPhi(s)\left(\frac{20}{5-6 s}-7\right)\D s.\end{split}
\end{align}In the last step, we note that $2(1+3s)$ becomes $5-6s $ as we trade $s$ for $\frac12-s$. Meanwhile, according to \eqref{eq:IKM151_MB} and \eqref{eq:IKM153_Cstar}, we have\begin{align}
\int_0^\infty I_0(t)[K_0(t)]^5t(1-8t^{2})\D t=-\frac{1}{3}\frac{1}{2\pi i}\int_{C_{*}}\varPhi(s)\left(\frac{20}{5-6 s}-7\right)\D s.
\end{align}Pairing up the last two displayed equations, we arrive at our destination.   \qedhere\end{enumerate}\begin{remark}
As we set  $a_0=n+\frac12,a_1=\frac{2}{3},a_2=n+\frac{1}{3} ,a_3=\frac{1}{3},a_4=n+\frac{1}{2},a_5=\frac{2}{3}-n$  in Zudilin's integral formula  \eqref{eq:Zudilin7F6}, we obtain\begin{align}\begin{split}&
\frac{2^{11/3} \pi ^{7/2}}{\left[\Gamma \left(\frac{1}{3}\right)\right]^4}\frac{\left[\Gamma \left(n+\frac{1}{3}\right)\right]^2 \Gamma \left(n+\frac{1}{2}\right) \Gamma \left(n+\frac{3}{2}\right)}{\Gamma \left(n+\frac{5}{6}\right) \Gamma \left(n+\frac{7}{6}\right) \Gamma \left(2 n+\frac{5}{6}\right)}{_7F_6}\left(\left. \begin{array}{c}
\frac{1}{3},\frac{2}{3},\frac{2}{3}-n,\frac{n}{2}+\frac{5}{4},n+\frac{1}{3},n+\frac{1}{2},n+\frac{1}{2} \\[4pt]
1,\frac{7}{6},\frac{n}{2}+\frac{1}{4},n+\frac{5}{6},n+\frac{7}{6},2 n+\frac{5}{6} \\
\end{array} \right|1\right) \\={}&\int_0^1 \int_0^1 \int_0^1\frac{[xz(1-z)]^{n}}{     \{1 - x [1 - y (1 - z)]\}^{2/3}}\frac{\D  x}{x^{2/3}\sqrt[6]{1-x}}\frac{\D y}{y^{2/3}\sqrt[3]{1-y} }\frac{\D z}{\sqrt{z}(1-z)^{2/3}}.\end{split}
\end{align}For $ n=0$ and $n=1$, we have just proved that the expression above  evaluates to \begin{align}2^{4/3}\sqrt{3}\frac{24\IKM(1,5;1)}{\pi}
\text{ and }2^{4/3}\sqrt{3}\frac{32\IKM(1,5;1)-256\IKM(1,5;3)}{21\pi},
\end{align}respectively. Numerically, we have also found that for small positive integers $n$, the last triple integral can be written in the following form:\begin{align}
2^{4/3}\sqrt{3}\frac{a_{n}\IKM(1,5;1)+b_n\IKM(1,5;3)}{\pi},\quad \text{where }a_n,b_n\in\mathbb Q.
\end{align} For example,\begin{align}
a_{2}=\frac{5359616}{24508575},\quad b_2=-\frac{47263744}{24508575}.
\end{align} Since the denominators of the rational numbers $a_n,b_n$ grow far too impetuously, we cannot use the formulations above to draw any definitive conclusion about the arithmetic nature for $ \IKM(1,5;1)$ or $ \IKM(1,5;3)$. We hope that some experts in Diophantine approximation will refine such identities in the future.
\eor\end{remark} \end{proof}
\subsection{Representations for $ \IKM(2,4;1)$ and $\IKM(2,4;3)$}
In Laporta's calculation of 4-loop contribution to electron's $g-2$ \cite[][(27)]{Laporta:2017okg}, the final result  did not involve the following  Feynman diagram with two pairs of  external legs\begin{align}
\;\;\;\;\;
\dia{\put(-50,0){\line(-1,-1){50}}
\put(-50,0){\line(-1,1){50}}
\put(50,0){\line(1,-1){50}}
\put(50,0){\line(1,1){50}}
\put(0,15){\circle{100}}
\put(0,-15){\circle{100}}
\put(50,0){\vtx}
\put(-50,0){\vtx}
}{}\;\;\;=2^{3}\int_0^\infty [I_0(t)]^{2}[K_0(t)]^4t\D t\equiv2^3\IKM(2,4;1),
\end{align}but this diagram did appear in the $ \varepsilon$-expansion of master integrals.

In the next two propositions, we will verify the following integral identity\begin{align}
\int_0^\infty [I_0(t)]^{2}[K_0(t)]^4t\D t=\frac{\pi^{2}}{30}\int_0^1{_2F_1}\left(\left.\begin{array}{c}
\frac{1}{3},\frac{2}{3}\ \\[4pt]
1 \\
\end{array}  \right|x\right){_2F_1}\left(\left.\begin{array}{c}
\frac{1}{3},\frac{2}{3}\ \\[4pt]
1 \\
\end{array}  \right|1-x\right)\frac{\D x}{\sqrt{1-x}}
\end{align}by turning both sides into special values of generalized hypergeometric series.
\begin{proposition}[Broadhurst--Laporta representations]\label{prop:IKM241_4F3}We have \begin{align}\begin{split}
\;\;\;\;\;
\dia{\put(-50,0){\line(-1,-1){50}}
\put(-50,0){\line(-1,1){50}}
\put(50,0){\line(1,-1){50}}
\put(50,0){\line(1,1){50}}
\put(0,15){\circle{100}}
\put(0,-15){\circle{100}}
\put(50,0){\vtx}
\put(-50,0){\vtx}
}{}\;\;\;={}&2^{3}\int_0^\infty [I_0(t)]^{2}[K_0(t)]^4t\D t\\={}&\frac{6 \pi ^{3/2}}{5}\left\{ \frac{\sqrt{3} }{2^6 }\left[\frac{\Gamma \left(\frac{1}{3}\right)}{\sqrt{\pi}}\right]^9\, _4F_3\left(\left. \begin{array}{c}\frac{1}{6},\frac{1}{3},\frac{1}{3},\frac{1}{2}\\[4pt]\frac{2}{3},\frac{5}{6},\frac{5}{6}\end{array} \right|1\right)+\frac{2^{4}}{3}\left[\frac{\sqrt{\pi}}{\Gamma \left(\frac{1}{3}\right)}\right]^9\, _4F_3\left(\left. \begin{array}{c}\frac{1}{2},\frac{2}{3},\frac{2}{3},\frac{5}{6}\\[4pt]\frac{7}{6},\frac{7}{6},\frac{4}{3}\end{array} \right|1\right) \right\}\\={}&\frac{4\pi^{2}}{5}{_4F_3}\left(\left.\begin{array}{c}
\frac{1}{3},\frac{1}{2},\frac{1}{2},\frac{2}{3} \\[4pt]\frac{5}{6},1,\frac{7}{6} \\
\end{array}\right|1\right),\end{split}\label{eq:IKM241_4F3}
\end{align}as indicated by Laporta \cite[][(27)]{Laporta:2017okg}  and Broadhurst (see \cite[][\S2.2]{Broadhurst2017Paris}, \cite[][\S2.2]{Broadhurst2017CIRM}, \cite[][\S2.1]{Broadhurst2017Higgs}, \cite[][\S3.1]{Broadhurst2017DESY}, \cite[][\S3.1]{Broadhurst2017ESIa}).
\end{proposition}\begin{proof}

Noting the Neumann addition theorems in \eqref{eq:I0_add} and \eqref{eq:K0_add},  we may integrate the Mellin--Barnes representation in  \eqref{eq:MB1}, and deduce\begin{align}\begin{split}&
\int_0^\infty [I_0(t)]^2[K_0(t)]^4t\D t\\={}&\frac{1}{80\sqrt{3}}\frac{1}{2\pi i}\int_{\frac14-i\infty}^{\frac14+i\infty}\frac{\Gamma \left(\frac{1}{3}-s\right) \Gamma \left(\frac{1}{2}-s\right) \Gamma \left(\frac{2}{3}-s\right) \Gamma \left(s-\frac{1}{6}\right) \Gamma (s) \Gamma \left(s+\frac{1}{6}\right)}{\Gamma (1-s) \Gamma \left(s+\frac{1}{2}\right)}\D s.\end{split}
\end{align}Closing the contour to the left, and collecting residues at all the simple poles, we may recast the expression above into\begin{align}\begin{split}&
\frac{3 \pi ^{3/2}}{10}\left\{ \frac{\sqrt{3} }{2^6 }\left[\frac{\Gamma \left(\frac{1}{3}\right)}{\sqrt{\pi}}\right]^9\, _4F_3\left(\left. \begin{array}{c}\frac{1}{6},\frac{1}{3},\frac{1}{3},\frac{1}{2}\\[4pt]\frac{2}{3},\frac{5}{6},\frac{5}{6}\end{array} \right|1\right)+\frac{2^{4}}{3}\left[\frac{\sqrt{\pi}}{\Gamma \left(\frac{1}{3}\right)}\right]^9\, _4F_3\left(\left. \begin{array}{c}\frac{1}{2},\frac{2}{3},\frac{2}{3},\frac{5}{6}\\[4pt]\frac{7}{6},\frac{7}{6},\frac{4}{3}\end{array} \right|1\right) \right\}\\{}&-\frac{\pi^{2}}{10}{_4F_3}\left(\left.\begin{array}{c}
\frac{1}{3},\frac{1}{2},\frac{1}{2},\frac{2}{3} \\[4pt]\frac{5}{6},1,\frac{7}{6} \\
\end{array}\right|1\right).\end{split}
\end{align}

We are almost done, except that we still need to verify  the last equality in  \eqref{eq:IKM241_4F3}. Towards this end,  we consider the following contour integral:\begin{align}
\frac{1}{2\pi i}\int_{\frac{1}{4}-i\infty}^{\frac{1}{4}+i\infty}\frac{\Gamma \left(\frac{1}{3}-s\right) \Gamma \left(\frac{1}{2}-s\right) \Gamma \left(\frac{2}{3}-s\right) \Gamma (s)}{\Gamma \left(\frac{5}{6}-s\right) \Gamma (1-s)\Gamma \left(\frac{7}{6}-s\right)  \Gamma \left(s+\frac{1}{2}\right)}\D s.
\end{align}Closing the contour leftwards, and summing over all the residues at $ s=-n,n\in\mathbb Z_{\geq0}$, we may evaluate the integral above as \begin{align}
2\sqrt{3}{_4F_3}\left(\left.\begin{array}{c}
\frac{1}{3},\frac{1}{2},\frac{1}{2},\frac{2}{3} \\[4pt]\frac{5}{6},1,\frac{7}{6} \\
\end{array}\right|1\right);
\end{align}closing the contour rightwards, we find that the total contributions from the  residues (\textit{i.e.} sum of all the residues, up to an overall minus sign) at  $ s=n+\frac{1}{3}$, $ s=n+\frac{1}{2}$ and $ s=n+\frac{2}{3}$ (for all $n\in\mathbb Z_{\geq0} $) to the contour integral are \begin{align}\frac{9}{16 \sqrt{\pi }}\left[\frac{\Gamma \left(\frac{1}{3}\right)}{\sqrt{\pi}}\right]^9\, &{_4F_3}\left(\left. \begin{array}{c}\frac{1}{6},\frac{1}{3},\frac{1}{3},\frac{1}{2}\\[4pt]\frac{2}{3},\frac{5}{6},\frac{5}{6}\end{array} \right|1\right),\\
-6\sqrt{3}&{_4F_3}\left(\left.\begin{array}{c}
\frac{1}{3},\frac{1}{2},\frac{1}{2},\frac{2}{3} \\[4pt]\frac{5}{6},1,\frac{7}{6} \\
\end{array}\right|1\right),\\\text{and }64 \sqrt{\frac{3}{\pi }}\left[\frac{\sqrt{\pi}}{\Gamma \left(\frac{1}{3}\right)}\right]^9&{ _4F_3}\left(\left. \begin{array}{c}\frac{1}{2},\frac{2}{3},\frac{2}{3},\frac{5}{6}\\[4pt]\frac{7}{6},\frac{7}{6},\frac{4}{3}\end{array} \right|1\right) ,
\end{align}respectively. Therefore, our goal is achieved.
\end{proof}
\begin{proposition}[Broadhurst integral]We have \begin{align}\begin{split}&
\int_0^1{_2F_1}\left(\left.\begin{array}{c}
\frac{1}{3},\frac{2}{3}\ \\[4pt]
1 \\
\end{array}  \right|x\right){_2F_1}\left(\left.\begin{array}{c}
\frac{1}{3},\frac{2}{3}\ \\[4pt]
1 \\
\end{array}  \right|1-x\right)\frac{\D x}{\sqrt{1-x}}\\={}&\frac{3 }{4 \sqrt{2} \pi ^2}G_{4,4}^{2,4}\left(1\left|
\begin{array}{c}
 \frac{1}{3},\frac{1}{2},\frac{1}{2},\frac{2}{3} \\[4pt]
 0,0,-\frac{1}{4},\frac{1}{4} \\
\end{array}
\right.\right)=3{_4F_3}\left(\left.\begin{array}{c}
\frac{1}{3},\frac{1}{2},\frac{1}{2},\frac{2}{3} \\[4pt]\frac{5}{6},1,\frac{7}{6} \\
\end{array}\right|1\right).\end{split}\label{eq:Broadhurst4F3}
\end{align}\end{proposition}\begin{proof}We paraphrase \cite[][(3.1.41)]{AGF_PartI} as follows:\begin{align}\begin{split}&
_2F_1\left(\left.\begin{array}{c}
-\nu,\nu+1\ \\[4pt]
1 \\
\end{array}  \right|x\right){_2F_1}\left(\left.\begin{array}{c}
-\nu,\nu+1\ \\[4pt]
1 \\
\end{array}  \right|1-x\right)\\={}&\frac{\sin^2(\nu\pi)}{\pi^2}\frac{1}{2\pi i}\int_{\delta-i\infty}^{\delta+i\infty}\frac{[\Gamma(s)]^{2}\Gamma(\nu+1-s)\Gamma(-\nu-s)\Gamma(\frac{1}{2}-s)}{\sqrt{\pi}\Gamma(1-s)}\frac{\D s}{[4x(1-x)]^s},\end{split}
\end{align}where $0<\delta<\min\{\nu+1,-\nu\},0<x<1 $. This allows us to compute\begin{align}\begin{split}&
\int_0^1{_2F_1}\left(\left.\begin{array}{c}
-\nu,\nu+1\ \\[4pt]
1 \\
\end{array}  \right|x\right){_2F_1}\left(\left.\begin{array}{c}
-\nu,\nu+1\ \\[4pt]
1 \\
\end{array}  \right|1-x\right)\frac{\D x}{\sqrt{1-x}}\\={}&\frac{\sin^2(\nu\pi)}{\sqrt{2} \pi ^2}\frac{1}{2\pi i}\int_{\delta-i\infty}^{\delta+i\infty}\frac{ \left[\Gamma \left(\frac{1}{2}-s\right)\right]^2 [\Gamma (s)]^2 \Gamma (-s-\nu ) \Gamma (-s+\nu +1)}{ \Gamma \left(\frac{3}{4}-s\right) \Gamma \left(\frac{5}{4}-s\right)}\D s\\={}&\frac{\sin^2(\nu\pi)}{\sqrt{2} \pi ^2}G_{4,4}^{2,4}\left(1\left|
\begin{array}{c}
 \frac{1}{2},\frac{1}{2},-\nu ,\nu +1 \\[4pt]
 0,0,-\frac{1}{4},\frac{1}{4} \\
\end{array}
\right.\right).\end{split}\end{align}Setting $\nu=-\frac13$ in the equation above, and $a= \frac{1}{2},b=\frac{1}{3},c=\frac{2}{3},d=\frac{1}{2},e=\frac{3}{4},f=\frac{1}{4}$ in Bailey's identity \eqref{eq:Bailey7F6}, we arrive at the last expression in \eqref{eq:Broadhurst4F3}. \end{proof}

In the next two propositions, we establish hypergeometric representations for $ \IKM(2,4;3)$, as stated in \eqref{eq:IKM243_hypergeo_repn}.\begin{proposition}[Mellin--Barnes integrals for $ \IKM(2,4;2n+1),n\in\{0,1,2\}$]\label{prop:IKM24n_MB}\begin{enumerate}[leftmargin=*,  label=\emph{(\alph*)},ref=(\alph*),
widest=a, align=left]\item Setting \begin{align}
\varPsi(s):=\frac{\Gamma \left(\frac{1}{3}-s\right) \Gamma \left(\frac{1}{2}-s\right) \Gamma \left(\frac{2}{3}-s\right) \Gamma \left(s-\frac{1}{6}\right) \Gamma (s) \Gamma \left(s+\frac{1}{6}\right)}{80\sqrt{3}\Gamma (1-s) \Gamma \left(s+\frac{1}{2}\right)}=\frac{9\varPhi(s)}{5\pi\sin(2\pi s)},
\end{align}we have {\allowdisplaybreaks\begin{align}\begin{split}
&\int_0^\infty [I_0(t)]^{2}[K_0(t)]^4t\D t\equiv\IKM(2,4;1)\\={}&\frac{1}{2\pi i}\int_{\frac14-i\infty}^{\frac14+i\infty}\varPsi(s)\D s,\end{split}\\\begin{split}{}&\int_0^\infty [I_0(t)]^{2}[K_0(t)]^4t^{3}\D t\equiv\IKM(2,4;3)\\={}&\frac{1}{2\pi i}\int_{\frac14-i\infty}^{\frac14+i\infty}\varPsi(s)\left[\frac{1}{3 (5-6 s)}+\frac{1}{2 s+1}-\frac{2}{3}\right]\D s+\frac{4 \pi ^{3/2}}{45}\left[\frac{\sqrt{\pi}}{\Gamma\left( \frac{1}{3} \right)}\right]^{9},\end{split}\\\begin{split}&\int_0^\infty [I_0(t)]^{2}[K_0(t)]^4t^{5}\D t\equiv\IKM(2,4;5)\\={}&\frac{1}{2\pi i}\int_{\frac14-i\infty}^{\frac14+i\infty}\varPsi(s)\left[\frac{25}{54 (7-6 s)}+\frac{43}{108 (5-6 s)}+\frac{23}{4 (2 s+1)}-\frac{45}{2 (2 s+3)}+\frac{68}{27} \right]\D s\\{}&+\frac{43 \pi ^{3/2}}{405}\left[\frac{\sqrt{\pi }}{\Gamma \left(\frac{1}{3}\right)}\right]^9-\frac{\pi ^{3/2}}{18432 \sqrt{3}}\left[\frac{\Gamma \left(\frac{1}{3}\right)}{\sqrt{\pi }}\right]^9.\end{split}
\end{align}}Moreover, we have  the following vanishing identity:\begin{align}\begin{split}0={}&
\frac{1}{2\pi i}\int_{\frac14-i\infty}^{\frac14+i\infty}\varPsi(s)\left[\frac{100}{3 (7-6 s)}+\frac{1}{3 (5-6 s)}+\frac{329}{2 s+1}-\frac{1620}{2 s+3}+240 \right]\D s\\&-\frac{\pi ^{3/2}}{256 \sqrt{3}}\left[\frac{\Gamma \left(\frac{1}{3}\right)}{\sqrt{\pi }}\right]^9+\frac{4 \pi ^{3/2}}{45}\left[\frac{\sqrt{\pi }}{\Gamma \left(\frac{1}{3}\right)}\right]^9.\end{split}\label{eq:Psi_vanish}
\end{align}\item We have\begin{align}
\frac{1}{2\pi i}\int_{\frac14-i\infty}^{\frac14+i\infty}\varPsi(s)\left(\frac{1}{5-6s}-\frac{2}{1+2s}+1\right)\D s=-\frac{4 \pi ^{3/2}}{15} \left[\frac{\sqrt{\pi}}{\Gamma\left( \frac{1}{3} \right)}\right]^{9},
\end{align}which entails\begin{align}\int_0^
\infty [I_0(t)]^{2}[K_0(t)]^4t^{3}\D t=\frac{1}{2\pi i}\int_{\frac14-i\infty}^{\frac14+i\infty}\varPsi(s)\left[\frac{5}{6(5-6s)}-\frac{1}{6}\right]\D s+\frac{2 \pi ^{3/2}}{9}\left[\frac{\sqrt{\pi}}{\Gamma\left( \frac{1}{3} \right)}\right]^{9}.
\end{align} \end{enumerate}\end{proposition}\begin{proof}The derivations of these formulae, in a similar vein as the proof of Proposition \ref{prop:IKM15n_MB}, are left to diligent readers.\end{proof}\begin{proposition}[Hypergeometric reduction of $\IKM(2,4;3) $]\label{prop:IKM24n_pFq}We have \begin{align}
\begin{split}&\int_0^\infty [I_0(t)]^{2}[K_0(t)]^4t(1-8t^2)\D t=\frac{7}{240 \sqrt{3}} G_{4,4}^{3,3}\left(1\left|
\begin{array}{c}
 -\frac{1}{2},\frac{2}{3},\frac{4}{3},\frac{1}{2} \\[4pt]
 -\frac{5}{6},-\frac{1}{6},1,0 \\
\end{array}
\right.\right)\\={}&\frac{7 \pi ^{3/2}}{60}\left\{\frac{\sqrt{3} }{2^7}\left[\frac{\Gamma \left(\frac{1}{3}\right)}{\sqrt{\pi }}\right]^9\, _4F_3\left(\left. \begin{array}{c}-\frac{1}{2},\frac{1}{6},\frac{1}{3},\frac{4}{3}\\[4pt]-\frac{1}{6},\frac{5}{6},\frac{5}{3}\end{array} \right|1\right)-\frac{5}{7}\frac{2^4}{3}  \left[\frac{\sqrt{\pi }}{\Gamma \left(\frac{1}{3}\right)}\right]^9\, _4F_3\left(\left. \begin{array}{c}-\frac{7}{6},-\frac{1}{2},-\frac{1}{3},\frac{2}{3}\\[4pt]-\frac{5}{6},\frac{1}{6},\frac{1}{3}\end{array} \right|1\right) \right\}\\={}&\frac{9\pi ^2}{550}  \, _4F_3\left(\left.\begin{array}{c}
\frac{2}{3},\frac{4}{3},\frac{3}{2},\frac{5}{2} \\[4pt]2,\frac{13}{6},\frac{17}{6} \\
\end{array}\right|1\right).\end{split}\label{eq:IKM243_pFq_G}
\end{align}\end{proposition}\begin{proof}Arguing as in Proposition \ref{prop:IKM15n_Meijer}(b), we have \begin{align}\begin{split}&
\int_0^\infty [I_0(t)]^{2}[K_0(t)]^4t(1-8t^2)\D t\\={}&-\frac{1}{3}\frac{1}{2\pi i}\int_{\frac14-i\infty}^{\frac14+i\infty}\varPsi(s)\left(\frac{20}{5-6 s}-7\right)\D s-\frac{16 \pi ^{3/2}}{9} \left[\frac{\sqrt{\pi}}{\Gamma\left( \frac{1}{3} \right)}\right]^{9}\\={}&-\frac{1}{2\pi i}\int_{\frac14-i\infty}^{\frac14+i\infty}\frac{\varPsi(s)}{3}\left(\frac{10}{5-6 s}+\frac{5}{1+3s}-7\right)\D s-\frac{16 \pi ^{3/2}}{9} \left[\frac{\sqrt{\pi}}{\Gamma\left( \frac{1}{3} \right)}\right]^{9}.\end{split}
\end{align}Checking the definition of  $ G_{4,4}^{3,3}$ against the integrand\begin{align}
\frac{\varPsi(s)}{3}\left(\frac{10}{5-6 s}+\frac{5}{1+3s}-7\right)=-\frac{7 \Gamma \left(-\frac{1}{3}-s\right) \Gamma \left(\frac{1}{3}-s\right) \Gamma \left(\frac{3}{2}-s\right) \Gamma \left(s-\frac{5}{6}\right) \Gamma \left(s-\frac{1}{6}\right) \Gamma (s+1)}{240 \sqrt{3} \Gamma (1-s) \Gamma \left(s+\frac{1}{2}\right)},
\end{align} we can verify the first equality in \eqref{eq:IKM243_pFq_G}.    Summing over all the residues of the last integrand at $ n-\frac13,n+\frac{1}{3},n+\frac{1}{2}$, where $ n\in\mathbb Z_{\geq0}$, we arrive at\begin{align}\begin{split}&
\int_0^\infty [I_0(t)]^{2}[K_0(t)]^4t(1-8t^2)\D t\\={}&\frac{7 \pi ^{3/2}}{30}\left\{\frac{\sqrt{3} }{2^7}\left[\frac{\Gamma \left(\frac{1}{3}\right)}{\sqrt{\pi }}\right]^9\, _4F_3\left(\left. \begin{array}{c}-\frac{1}{2},\frac{1}{6},\frac{1}{3},\frac{4}{3}\\[4pt]-\frac{1}{6},\frac{5}{6},\frac{5}{3}\end{array} \right|1\right)-\frac{5}{7}\frac{2^4}{3}  \left[\frac{\sqrt{\pi }}{\Gamma \left(\frac{1}{3}\right)}\right]^9\, _4F_3\left(\left. \begin{array}{c}-\frac{7}{6},-\frac{1}{2},-\frac{1}{3},\frac{2}{3}\\[4pt]-\frac{5}{6},\frac{1}{6},\frac{1}{3}\end{array} \right|1\right) \right\}\\{}&-\frac{9\pi ^2}{550}  \, _4F_3\left(\left.\begin{array}{c}
\frac{2}{3},\frac{4}{3},\frac{3}{2},\frac{5}{2} \\[4pt]2,\frac{13}{6},\frac{17}{6} \\
\end{array}\right|1\right).\end{split}
\end{align}

Similar to what we did  in proof of Proposition \ref{prop:IKM241_4F3}, we evaluate the  following contour integral\begin{align}\frac1{2\pi i}\int_{\frac14-i\infty}
^{\frac14+i\infty}\frac{ \Gamma \left(-\frac{1}{3}-s\right) \Gamma \left(\frac{1}{3}-s\right) \Gamma \left(\frac{3}{2}-s\right) \Gamma (s+1)}{ \Gamma (1-s) \Gamma \left(s+\frac{1}{2}\right) \Gamma \left(\frac{7}{6}-s\right) \Gamma \left(\frac{11}{6}-s\right)}\D s
\end{align} in two ways, to verify the last equality in  \eqref{eq:IKM243_pFq_G}.  Thus, all the relations in   \eqref{eq:IKM243_pFq_G} are true.
\end{proof}

\subsection*{Acknowledgments}This research was supported in part  by the Applied Mathematics Program within the Department of Energy
(DOE) Office of Advanced Scientific Computing Research (ASCR) as part of the Collaboratory on
Mathematics for Mesoscopic Modeling of Materials (CM4).

A large proportion of this work has been assembled from my research notes on hypergeometric series, which were prepared at Princeton in 2012. I thank Prof.\ Weinan E (Princeton University and Peking University) for running a seminar on  mathematical problems in quantum fields  at Princeton, covering both 2-dimensional and $(4-\varepsilon)$-dimensional theories.

I am grateful to Dr.\ David Broadhurst for many fruitful communications on recent progress in the arithmetic properties of Feynman diagrams. In particular, I thank him for suggesting the challenging integral identity in \eqref{eq:IKM151_IKM153_diff}.


\end{document}